\documentclass[12pt,a4paper,twoside,reqno]{amsart}
\usepackage[a4paper,top=3cm,bottom=3cm,inner=2.5cm,outer=2.5cm]{geometry}

\usepackage[utf8]{inputenc}
\usepackage{amsmath,amssymb,amsthm,amsfonts}
\usepackage[english]{babel}
\usepackage{multirow,enumitem,array}
\usepackage{xparse}
\usepackage{pst-node,placeins,xspace,fix-cm}
\usepackage{tikz-cd}
\usepackage{accents}
\usepackage[nomessages]{fp}
\usepackage[colorlinks,linkcolor=blue,citecolor=violet,urlcolor=darkgray,final,hyperindex,linktoc=page,hyperfootnotes=true]{hyperref}
\newcolumntype{F}{>{$}c<{\hspace{-0.9ex}$}}
\newcolumntype{:}{>{$}m{0.8ex}<{$}}
\newcolumntype{R}{>{$}r<{$}}
\newcolumntype{C}{>{$}c<{$}}
\newcolumntype{L}{>{$}l<{$}}
\newcolumntype{N}{@{}>{$}l<{$}}
\newcommand{\linesep}[1]{\renewcommand{\arraystretch}{#1}}
\newlength\horspace
\newcommand{\h}[1][1.0]{\hspace{#1\horspace}}
\newlength\verspace
\renewcommand{\v}[1][1.0]{\vspace{#1\verspace}\xspace}
\newlength\negverspace
\makeatletter
\newcommand\semiHuge{\@setfontsize\semiHuge{22.72}{27.38}}
\newcommand\semiLarge{\@setfontsize\semiLarge{12}{14}}
\makeatother
\tikzset{iso/.style={draw=none,every to/.append style={edge node={node [sloped, allow upside down, auto=false]{$\cong$}}}}}
\tikzset{adjunction/.style={draw=none,every to/.append style={edge node={node [sloped, allow upside down, auto=false]{$\dashv$}}}}}
\tikzset{simeq/.style={draw=none,every to/.append style={edge node={node [sloped, allow upside down, auto=false]{$\simeq$}}}}}
\tikzset{simeqS/.style={draw=none,every to/.append style={edge node={node [sloped, allow upside down, auto=false]{$\raisebox{0.8em}{$\simeq$}$}}}}}
\tikzset{aiso/.style={simeqS,preaction={draw,->}}}
\tikzset{dotdot/.style={dash pattern=on 0.25ex off 0.2ex, dash phase=0ex}}
\tikzset{RightA/.style={double distance=3.5pt,>={Implies},->},%
	triple/.style={-,preaction={draw,RightA}},%
	quadruple/.style={preaction={draw,RightA,shorten >=0pt},shorten >=1pt,-,double,double distance=0.2pt}}

\newtheorem{teor}{Theorem}[section]

\newtheorem{prop}[teor]{Proposition}

\theoremstyle{definition}
\newtheorem{defne}[teor]{Definition}

\newtheorem{rem}[teor]{Remark}
\newtheorem{exampl}[teor]{Example}

\newtheorem{cons}[teor]{Construction}

\def\nameit#1{\textrm{#1}~}
\def\defx{\nameit{Definition}}
\def\thex{\nameit{Theorem}}
\def\prox{\nameit{Proposition}}
\def\conx{\nameit{Construction}}
\def\remx{\nameit{Remark}}

\def\exax{\nameit{Example}}

\NewDocumentEnvironment{cd}{s O{7} O{7} b}{%
	\IfBooleanF{#1}{\begin{equation*}}\begin{tikzcd}[row sep=#2ex,column sep=#3ex,ampersand replacement=\&]
			#4
		\end{tikzcd}\IfBooleanF{#1}{\end{equation*}}\ignorespacesafterend}{}

\newenvironment{enum}{\begin{enumerate}[label=$($\hspace{0.12ex}\roman*\hspace{0.075ex}$)$]}{\end{enumerate}}
\newenvironment{enumT}{\begin{enumerate}[label=$($\hspace{-0.1ex}\roman*\hspace{0.13ex}$)$]}{\end{enumerate}}
\newenvironment{fun}{\[\begin{tabular}{F:RCL}}{\end{tabular}\]\ignorespacesafterend}

\newenvironment{eqD*}{\begin{equation*}}{\end{equation*}\ignorespacesafterend}

\def\:{\colon}
\providecommand\ordinarycolon{:}
\def\vcentcolon{\mathrel{\mathop\ordinarycolon}}
\newcommand{\deq}{\mathrel{\vcentcolon\mkern-1.2mu=}}
\def\phi{\varphi}

\def\O{\ensuremath{\Omega}\xspace}

\newcommand{\romanuppercase}[1]{\uppercase\expandafter{\romannumeral #1\relax}}

\newcommand{\p}[1]{\big(\mkern1mu{#1}\mkern1mu\big)}
\newcommand{\pteor}[1]{$($\h[-1.9]{#1}\h[0.6]$)$}

\def\dfn#1{{\bfseries\itshape #1}}
\def\predfn#1{{\itshape #1}}

\newcommand{\scaleu}[2][1.2]{{\scalebox{#1}{$#2$}}}

\newcommand{\ov}[1]{\overline{#1}}
\def\t#1{\widetilde{#1}} 
\def\b{_{\bullet}}
\def\bl{_{\bullet,\operatorname{lax}}}

\def\c{\circ}
\newcommand{\cont}{\subseteq}

\newcommand{\st}{^{\ast}}
\newcommand{\stb}{_{\ast}}

\newcommand{\slant}[2]{{\raisebox{.05em}{$#1$}\mkern-1mu\left/\raisebox{-.15em}{$#2$}\right.}}
\newcommand{\sliceslant}[2]{{\raisebox{.1em}{$#1$}\mkern-1mu\left/\raisebox{-.25em}{$#2$}\right.}}

\newcommand{\oplaxsliceslant}[2]{{\raisebox{.1em}{$#1$}\mkern-1mu\left/_{\mkern-4.1mu\oplax}\right.\raisebox{-.25em}{$#2$}}}

\makeatletter
\newcommand{\douwidehat}[2]{%
	\sbox0{$\m@th#1\widehat{\hphantom{#2}}$}%
	\sbox2{$\m@th#1x$}
	\sbox4{$\m@th#1#2$}
	\dimen0=\ht0
	\advance\dimen0 -.8\ht2
	\dimen2=\dp4
	\rlap{%
		\raisebox{\dimexpr\dimen0-\dimen2}{%
			\scalebox{1}[-1]{\box0}}}{#2}}
\makeatother

\DeclareFontFamily{OT1}{pzc}{}
\DeclareFontShape{OT1}{pzc}{m}{it}{<->s*[1.21]pzcmi7t}{}
\DeclareMathAlphabet{\mathpzc}{OT1}{pzc}{m}{it}

\DeclareFontFamily{U}{dutchcal}{\skewchar\font=45}
\DeclareFontShape{U}{dutchcal}{m}{n}{<->s*[1.05] dutchcal-r}{}
\DeclareMathAlphabet{\mathlcal}{U}{dutchcal}{m}{n}
\newcommand{\catfont}[1]{\ensuremath{\mathpzc{#1}}\xspace}

\newcommand{\A}{\catfont{A}}
\newcommand{\B}{\catfont{B}}
\newcommand{\C}{\catfont{C}}
\newcommand{\D}{\catfont{D}}
\newcommand{\E}{\catfont{E}}
\newcommand{\F}{\catfont{F}}

\newcommand{\V}{\catfont{V}}

\newcommand{\U}{\catfont{U}}
\newcommand{\W}{\catfont{W}}
\def\L{\catfont{L}}

\newcommand{\I}{\ensuremath{\mathcal{I}}\xspace}
\newcommand{\1}{\catfont{1}}
\newcommand{\2}{\catfont{2}}

\newcommand{\Set}{\catfont{Set}}

\newcommand{\Cat}{\catfont{Cat}}
\newcommand{\Catzero}{\ensuremath{\Cat_{0}}\xspace}
\newcommand{\CAT}{\catfont{Cat}}
\newcommand{\CATlarge}{\catfont{CAT}}
\newcommand{\twoCAT}{2\h[2]\mbox{-}\CAT}
\newcommand{\twoCATlax}{2\h[2]\mbox{-}\CAT_{\h[-3]\lax}}

\NewDocumentCommand{\Fib}{t' t" t+ t? O{n} O{n} o}{
	\ensuremath{\catfont{\ifx#5d{D}\else{\ifx#5o{Op}\else{\ifx#5b{DOp}\else{\ifx#5t{D2Op}\else{\ifx#5c{Cl\h[-3]}\else{\ifx#5s{Sp}\fi}\fi}\fi}\fi}\fi}\fi{Fib}}\IfBooleanT{#3}{^{\h[3.7]\opn{s}\h[-3.2]}}\IfBooleanT{#4}{^{\h[3.7]\opn{P}\h[-3,7]}}{\IfBooleanTF{#1}{_{\h[0.4]\opn{cart}\ifx#6n{}\else{\h[-1.4],\h[0.4]{#6}}\fi}}{\IfBooleanTF{#2}{_{\h[0.4]\opn{clov}\ifx#6n{}\else{\h[-1.4],\h[0.4]{#6}}\fi}}{\ifx#6n{}\else{_{\h[0.4]{#6}}}\fi}}}\IfNoValueF{#7}{\h[-1]\left({#7}\right)}}\xspace
}

\NewDocumentCommand{\Sh}{o m}{
	\ensuremath{\catfont{Sh}\hspace{-0.15ex}\left({#2}\IfNoValueF{#1}{,{#1}}\right)}
}

\NewDocumentCommand{\St}{t+ o m}{
	\ensuremath{\catfont{St}\IfBooleanT{#1}{_{\opn{strict}}}\hspace{-0.15ex}\left({#3}\IfNoValueF{#2}{,{#2}}\right)}
}

\NewDocumentCommand{\Alg}{t+ t' m}{
	\ensuremath{\IfBooleanT{#1}{\catfont{Ps}\mbox{-}}{#3}\mbox{-}\catfont{\IfBooleanT{#2}{Co}Alg}}
}

\newcommand{\slice}[2]{\sliceslant{#1}{#2}}

\newcommand{\oplaxslice}[2]{\oplaxsliceslant{#1}{#2}}

\newcommand{\HomC}[3]{{#1}\left({#2},\h[1]{#3}\right)}
\newcommand{\m}[2]{\ensuremath{\left[#1,#2\right]}\xspace}

\newcommand{\mlax}[2]{\ensuremath{\left[#1,#2\right]_{\lax}}\xspace}
\newcommand{\moplax}[2]{\ensuremath{\left[#1,#2\right]_{\oplax}}\xspace}
\newcommand{\mPsoplax}[2]{\ensuremath{\opn{Ps}\left[#1,#2\right]_{\oplax}}\xspace}

\newcommand{\Sub}[1]{\operatorname{Sub}\hspace{-0.1ex}\left({#1}\right)}

\newcommand{\opn}[1]{\operatorname{#1}}
\newcommand{\y}[1]{\ensuremath{\operatorname{y}\hspace{-0.2ex}\left({#1}\right)}}

\DeclareMathOperator{\yy}{y}
\newcommand{\yyop}{\yy\op}
\newcommand{\id}[1]{\operatorname{id}_{#1}}
\newcommand{\Id}[1]{\operatorname{Id}_{#1}}

\newcommand{\op}{\ensuremath{^{\operatorname{op}}}}
\newcommand{\co}{\ensuremath{^{\operatorname{co}}}}

\newcommand{\restr}[2]{{\left.\kern-\nulldelimiterspace {#1}\vphantom{\big|} \right|_{#2}}}

\newcommand{\dom}{\operatorname{dom}}

\DeclareMathOperator{\colim}{colim}
\DeclareMathOperator{\lax}{lax}
\DeclareMathOperator{\oplax}{oplax}

\DeclareMathOperator{\pseudo}{pseudo}

\newcommand{\wcolim}[2]{{\colim}^{#1}\h{#2}}

\newcommand{\oplaxcolim}[2][]{\oplax\mbox{-}\h[1.5]\wcolim{#1}{#2}}

\newcommand{\G}[1]{\catfont{G}_{#1}}
\newcommand{\Gprime}[1]{\catfont{G}'_{#1}}

\newcommand{\Int}[1]{\ensuremath{\int \hspace{-0.35ex} #1}}
\newcommand{\Intdiag}[1]{\ensuremath{\scaleu{\int} \hspace{-0.15ex} #1}}
\newcommand{\Intop}[1]{\ensuremath{\int\op \hspace{-0.35ex} #1}}

\newcommand{\Intopdiag}[1]{\ensuremath{\scaleu{\int}\op \hspace{-0.15ex} #1}}
\newcommand{\Groth}[1]{\Int{#1}}
\newcommand{\Grothdiag}[1]{\Intdiag{#1}}
\newcommand{\Grothop}[1]{\Intop{#1}}

\newcommand{\Grothopdiag}[1]{\Intopdiag{#1}}
\newcommand{\groth}[1]{\mathcal{G}\mkern-1.4mu\left(#1\right)}
\newcommand{\grothprime}[1]{\mathcal{G}'\mkern-1.4mu\left(#1\right)}

\newcommand{\too}{\longrightarrow}
\newcommand{\mto}{\mapsto}

\newcommand{\ito}{\hookrightarrow}

\makeatletter
\newcommand{\ar}[2][]{\xrightarrow[#1]{#2}}
\def\xlongrightarrowfill@{\arrowfill@\relbar\relbar\longrightarrow}
\newcommand{\arr}[2][]{%
	\ext@arrow 0099\xlongrightarrowfill@{#1}{#2}}
\newcommand{\aarr}[2][]{%
	\ext@arrow 0099\xlongrightarrowfill@{#1}{#2}} 

\newcommand{\aR}[2][]{%
	\ext@arrow 0055{\Rightarrowfill@}{#1}{#2}}
\def\xLongrightarrowfill@{\arrowfill@\Relbar\Relbar\Longrightarrow}
\newcommand{\aRR}[2][]{%
	\ext@arrow 0099\xLongrightarrowfill@{#1}{#2}}
\def\aitofill@{\arrowfill@{\lhook\joinrel\relbar}\relbar\rightarrow}
\newcommand{\aito}[2][]{%
	\ext@arrow 3095\aitofill@{#1}{#2}}
\def\Longaitofill@{\arrowfill@{\lhook\joinrel\relbar\joinrel\relbar}\relbar\rightarrow}
\newcommand{\aitoo}[2][]{%
	\ext@arrow 0099\Longaitofill@{#1}{#2}}

\def\xlongleftarrowfill@{\arrowfill@\longleftarrow\relbar\relbar}
\newcommand{\all}[2][]{%
	\ext@arrow 0099\xlongleftarrowfill@{#1}{#2}}
\newcommand{\aL}[2][]{%
	\ext@arrow 0055{\Leftarrowfill@}{#1}{#2}}
\def\xLongleftarrowfill@{\arrowfill@\Longleftarrow\Relbar\Relbar}
\newcommand{\aLL}[2][]{%
	\ext@arrow 0099\xLongleftarrowfill@{#1}{#2}}
\def\xmapstofill@{\arrowfill@{\mapstochar\relbar}\relbar\rightarrow}
\newcommand{\am}[2][]{%
	\ext@arrow 0395\xmapstofill@{#1}{#2}}
\def\xlongmapstofill@{\arrowfill@\relbar\relbar\longmapsto}
\newcommand{\amm}[2][]{%
	\ext@arrow 0399\xlongmapstofill@{#1}{#2}}

\newcommand{\eqq}{\DOTSB\protect\Relbar\protect\joinrel\Relbar}
\def\xeqqfill@{\arrowfill@\Relbar\Relbar\eqq}
\newcommand{\aeqq}[2][]{%
	\ext@arrow 0099\xeqqfill@{#1}{#2}}
\def\xRrightarrowfill@{\arrowfill@\equiv\equiv\Rrightarrow}
\newcommand{\aM}[2][]{\ext@arrow 0359\xRrightarrowfill@{#1}{#2}}
\newcommand{\Llongrightarrow}{%
	\DOTSB\protect\equiv\protect\joinrel\Rrightarrow}
\def\xLlongrightarrowfill@{\arrowfill@\equiv\equiv\Llongrightarrow}
\newcommand{\aMM}[2][]{%
	\ext@arrow 0099\xLlongrightarrowfill@{#1}{#2}}
\makeatother
\newcommand{\alax}[1]{\aR[\lax]{#1}}

\newcommand{\aoplax}[1]{\aR[\oplax]{#1}}

\newcommand{\apseudo}[1]{\aR[\pseudo]{#1}}

\newcommand{\iso}{\cong}

\newcommand{\aequi}{\ensuremath{\stackrel{\raisebox{-1ex}{\kern-.3ex$\scriptstyle\sim$}}{\rightarrow}}}
\newcommand{\aequii}{\ensuremath{\stackrel{\raisebox{-1ex}{\kern-.3ex$\scriptstyle\sim$}}{\longrightarrow}}}

\newcommand{\PB}[1]{\arrow[#1,phantom,"\scalebox{1.6}{\color{black}$\lrcorner$}",very near start]}
\newcommand{\Ar}[4][]{\arrow[#2,"{#3}"{#1},""{name=#4, anchor=center}]}
\newcommand{\Ars}[4][]{\arrow[#2,"{#3}"'{#1},""{name=#4, anchor=center}]}
\newcommand{\Arb}[6][]{\arrow[#2,"{#3}"{#1},from=#4,to=#5,shorten <= #6 em, shorten >= #6 em]}
\newcommand{\Arbs}[6][]{\arrow[#2,"{#3}"'{#1},from=#4,to=#5,shorten <= #6 em, shorten >= #6 em]}

\NewDocumentCommand{\fib}{O{n} O{2.3} mmm}{%
	\begin{cd}*[#2][5]
		{#3}\ifx#1n{\arrow[d,"{\,\scaleu{#4}}"]}\else{\ifx#1i{\arrow[d,hookrightarrow,"{\,\scaleu{#4}}"]}\else{\ifx#1e{\arrow[d,equal,"{\,\scaleu{#4}}"]}\else{\ifx#1R{\arrow[d,Rightarrow,"{\,\scaleu{#4}}"]}\fi}\fi}\fi}\fi\\
		{#5}\ifx#1o{\arrow[u,"{\,\scaleu{#4}}"']}\fi
	\end{cd}\xspace
}
\NewDocumentCommand{\fibdiag}{O{n} O{2.3} mmm}{%
	\begin{cd}*[#2][5]
		{#3}\ifx#1n{\arrow[d,"{\,{#4}}"]}\else{\ifx#1i{\arrow[d,hookrightarrow,"{\,{#4}}"]}\else{\ifx#1e{\arrow[d,equal,"{\,{#4}}"]}\else{\ifx#1R{\arrow[d,Rightarrow,"{\,{#4}}"]}\fi}\fi}\fi}\fi\\
		{#5}\ifx#1o{\arrow[u,"{\,{#4}}"']}\fi
	\end{cd}\xspace
}
	
\NewDocumentCommand{\sq}{s O{n} O{7} O{7} O{} O{2.7} O{2.2} O{0.5} O{n}}{%
	\def\foosq##1##2##3##4##5##6##7##8{%
		\IfBooleanTF{#1}{\begin{cd}*}{\begin{cd}}[#3][#4]
				{##1}\ifx#2p{\PB{rd}}\fi\arrow[r,"{##5}"]\ifx#9l{\arrow[d,equal,"{##6}"']}\else{\arrow[d,"{##6}"']}\fi\&{##2}\ifx#9r{\arrow[d,equal,"{##7}"]}\else{\arrow[d,"{##7}"]}\fi\ifx#2l{\arrow[ld,Rightarrow,shorten <=#6ex,shorten >=#7ex,"{#5}"{pos=#8}]}\fi\\
				{##3}\ifx#9d{\arrow[r,equal,"{##8}"']}\else{\arrow[r,"{##8}"']}\fi\ifx#2o{\arrow[ur,Rightarrow,shorten <=#6ex,shorten >=#7ex,"{#5}"{pos=#8}]}\fi\&{##4}
		\end{cd}}%
		\foosq}
\NewDocumentCommand{\sqs}{s O{n} O{7} O{7} O{} O{} O{} O{}}{%
	\def\foosqs##1##2##3##4##5##6##7##8{%
		\IfBooleanTF{#1}{\begin{cd}*}{\begin{cd}}[#3][#4]
				{##1}\ifx#2p{\PB{rd}}\fi\arrow[r,"{##5}"#5]\arrow[d,"{##6}"'#6]\&{##2}\arrow[d,"{##7}"#7]\\
				{##3}\arrow[r,"{##8}"'#8]\&{##4}
		\end{cd}}%
		\foosqs}
\NewDocumentCommand{\nat}{s O{n} O{7} O{7} O{2.7} O{2.2} O{0.5} O{n}}{%
	\def\foonat##1##2##3##4##5##6{%
		\IfBooleanTF{#1}{\sq*}{\sq}[#2][#3][#4][{##1}_{##4}][#5][#6][#7][#8]{{##2}\ifx#8l{}\else{({##5})}\fi}{{##3}\ifx#8r{}\else{({##5})}\fi}{{##2}\ifx#8l{}\else{({##6})}\fi}{{##3}\ifx#8r{}\else{({##6})}\fi}{{##1}_{##5}}{\ifx#8l{}\else{{##2}({##4})}\fi}{\ifx#8r{}\else{{##3}({##4})}\fi}{{##1}_{##6}}}%
	\foonat}
\NewDocumentCommand{\tr}{s O{4.5} O{6.5} O{0} O{0} O{n} O{0} O{} O{0}}{%
	\def\footr##1##2##3##4##5##6{%
		\IfBooleanTF{#1}{\begin{cd}*}{\begin{cd}}[#3][#2]
				{##1}\arrow[rr,"{##4}"]
				\Ars[inner sep =0.2ex]{dr}{##5}{A}\&[#4ex]\&[#5ex]{##2}\Ar[inner sep =0.2ex]{ld}{##6}{B}\\
				\&{##3}
				\ifx#6l{\Arb{Rightarrow,shift right=#7em}{#8}{A}{B}{#9}}\else{\ifx#6o{\Arbs{Rightarrow,shift right=#7em}{#8}{B}{A}{#9}}\else{\ifx#6i{\Arbs[inner sep=0.9ex]{iso,shift right=#7em}{#8}{A}{B}{#9}}\else{\ifx#6e{\Arb{equal,shift right=#7em}{#8}{A}{B}{#9}}\else{}\fi}\fi}\fi}\fi
		\end{cd}}%
		\footr}
\NewDocumentCommand{\trslice}{s t+ O{7} O{7}}{%
	\def\footrslice##1##2##3##4##5##6{%
		\IfBooleanTF{#1}{\begin{cd}*}{\begin{cd}}[#3][#4]
				{##1}\arrow[r,"{##4}"]\IfBooleanF{#2}{\arrow[d,"{##5}"']}\&{##2}\\
				{##3}\IfBooleanT{#2}{\arrow[u,"{##5}"]}\arrow[ru,"{##6}"']
		\end{cd}}%
		\footrslice}
\NewDocumentCommand{\tc}{s t+ O{7} O{30} O{} O{} O{} o}{
	\def\footc##1##2##3##4##5{%
		\FPmul\Mulresulttwo{#3}{#3}%
		\FPmul\Mulresult{0.0026}{\Mulresulttwo}%
		\IfBooleanTF{#1}{\begin{cd}*}{\begin{cd}}[#3][#3]
				{##1}\Ar[#5]{r,bend left=#4}{##3}{A}\Ars[#6]{r,bend right=#4}{##4}{B}\&{##2}
				\IfBooleanTF{#2}{\Arb[description,pos=0.49]}{\Arb}{Rightarrow #7}{\mkern1mu {##5}}{A}{B}{\IfNoValueTF{#8}{\Mulresult}{#8}}
		\end{cd}}%
		\footc}
\NewDocumentCommand{\tcwl}{s t+ O{7} O{30} O{} O{} O{} o O{-2}}{
	\def\footcwl##1##2##3##4##5##6##7{%
		\FPmul\Mulresulttwo{#3}{#3}%
		\FPmul\Mulresult{0.0026}{\Mulresulttwo}%
		\IfBooleanTF{#1}{\begin{cd}*}{\begin{cd}}[#3][#3]
				##6 \arrow[r,"{##7}"]\&[#9ex]{##1}\Ar[#5]{r,bend left=#4}{##3}{A}\Ars[#6]{r,bend right=#4}{##4}{B}\&{##2}\IfBooleanTF{#2}{\Arb[description,pos=0.49]}{\Arb}{Rightarrow #7}{\mkern1mu {##5}}{A}{B}{\IfNoValueTF{#8}{\Mulresult}{#8}}
		\end{cd}}%
		\footcwl}
\NewDocumentCommand{\tcwr}{s t+ O{7} O{30} O{} O{} O{} o O{-2}}{
	\def\footcwr##1##2##3##4##5##6##7{%
		\FPmul\Mulresulttwo{#3}{#3}%
		\FPmul\Mulresult{0.0026}{\Mulresulttwo}%
		\IfBooleanTF{#1}{\begin{cd}*}{\begin{cd}}[#3][#3]
				{##1}\Ar[#5]{r,bend left=#4}{##3}{A}\Ars[#6]{r,bend right=#4}{##4}{B}\&{##2}\arrow[r,"{##7}"]\&[#9ex]##6\IfBooleanTF{#2}{\Arb[description,pos=0.49]}{\Arb}{Rightarrow #7}{\mkern1mu {##5}}{A}{B}{\IfNoValueTF{#8}{\Mulresult}{#8}}
		\end{cd}}%
		\footcwr}
\NewDocumentCommand{\tcv}{s t' O{7} O{30} mmmmm}{
	\FPmul\Mulresulttwo{#3}{#3}%
	\FPmul\Mulresult{0.0026}{\Mulresulttwo}%
	\IfBooleanTF{#1}{\begin{cd}*}{\begin{cd}}[#3][#3]
			{#5}\IfBooleanTF{#2}{\Ars{d,leftarrow,bend right=#4}{#7}{A}\Ar{d,leftarrow,bend left=#4}{#8}{B}}{\Ars{d,bend right=#4}{#7}{A}\Ar{d,bend left=#4}{#8}{B}}\\{#6}
			\Arb{Rightarrow}{#9}{A}{B}{\Mulresult}
		\end{cd}}
\NewDocumentCommand{\twonats}{s O{2.2} O{8} O{7} O{1.05} O{3.45} O{2}}{%
	\def\footwonats##1##2##3##4##5##6##7##8##9{%
		\def\foofootwonats####1####2####3####4####5{%
			\IfBooleanTF{#1}{\begin{cd}*}{\begin{cd}}[#3][#4]
					##1 \Ar{r}{##9}{} \Ars{d,bend right=40}{##5}{A} \Ar{d,bend left=40}{##6}{B} \&
					##2 \Ars{d,bend left}{##8}{Q} \arrow[ld,Rightarrow,shift left=#7,"{####4}"{pos=0.48},shorten <=#5ex, shorten >=#6ex]\&[-2ex]
					##1 \Ar{r}{##9}{} \Ar{d,bend right}{##5}{R} \&
					##2 \Ars{d,bend right=40}{##7}{C} \Ar{d,bend left=40}{##8}{D} \arrow[ld,Rightarrow,shift right=#7,"{####5}"'{pos=0.52},shorten <=#6ex, shorten >=#5ex] \\
					##3 \Ars{r}{####1}{} \&
					##4 \&
					##3 \Ars{r}{####1}{} \& 
					##4
					\Arbs{Rightarrow}{\,{####2}}{B}{A}{0.3}
					\Arbs{Rightarrow}{\,{####3}}{D}{C}{0.3}
					\Arb{equal}{}{Q}{R}{#2}
			\end{cd}}%
			\foofootwonats}\footwonats}
\makeatletter
\newcommand{\DOpFiboi}[2][\@nil]{%
	\def\tmp{#1}%
	\ifx\tmp\@nnil{\ensuremath{\slant{\catfont{DOpFib}\left({#2}\right)}}{\iso}}%
	\else{\ensuremath{\slant{\catfont{DOpFib}_{\h[0.4]{#1}}\h[-1]\left({#2}\right)}{\iso}}}\fi}
\makeatother

\newcommand{\pbsqunivv}[9][]{%
	\def\foopbsqunivv##1##2##3##4{%
		\begin{cd}[5][4.5]
			#2\arrow[rrd,bend left=24,"{#3}"{inner sep=0.35ex}]\arrow[rdd,bend right=37,"{#4}"']\arrow[rd,dashed,"{#5}"{#1},shorten <=-0.75ex,shorten >=-0.3ex]\&[-4ex]\\[-4.4ex]
			\&#6 \arrow[r,"{##1}"] \arrow[d,"{##2}"'] \PB{rd} \& #7 \arrow[d,"{##3}"] \\
			\&#8 \arrow[r,"{##4}"'] \& #9
	\end{cd}}%
	\foopbsqunivv%
}

\newcommand{\fibsq}[8][6]{%
	\begin{cd}[5.85][#1]
		#2 \arrow[r,"{#6}"] \arrow[d,mapsto,"{#7}"'] \& #3 \arrow[d,mapsto,"{#7}"] \\
		#4 \arrow[r,"{#8}"'] \& #5
\end{cd}}
\newcommand{\fibsqN}[8][6]{%
	\begin{cd}*[5.3][#1]
		#2 \arrow[r,"{#6}"] \arrow[d,mapsto,"{#7}"'] \& #3 \arrow[d,mapsto,"{#7}"] \\
		#4 \arrow[r,"{#8}"'] \& #5
\end{cd}}
\newcommand{\fibsquniv}[8][6]{%
	\def\foofibsquniv##1##2##3##4##5{%
		\begin{cd}[3][#1]
			\& \&[-3ex] ##1 \arrow[dd,mapsto,"{#7}"]\\[-4.5ex]
			#2 \arrow[rru,bend left,"{##3}"]\arrow[r,"{#6}"] \arrow[dd,mapsto,"{#7}"'] \& #3\arrow[ru,dashed,"{##4}"] \arrow[dd,mapsto,"{#7}"] \\
			\& \& ##2 \\[-4.5ex]
			#4 \arrow[r,"{#8}"'] \& #5 \arrow[ru,"{##5}"']
			\end{cd}}%
	\foofibsquniv%
}
			
\begin{document}
				
\title{Indexed Grothendieck construction}
\author[E. Caviglia]{Elena Caviglia}
\address{School of Computing and Mathematical Sciences, University of Leicester, United Kingdom}
\email{ec363@leicester.ac.uk}
\author[L. Mesiti]{Luca Mesiti}
\address{School of Mathematics, University of Leeds, United Kingdom}
\email{mmlme@leeds.ac.uk}

\begin{abstract}
	We produce an indexed version of the Grothendieck construction. This gives an equivalence of categories between opfibrations over a fixed base in the 2-category of 2-copresheaves and 2-copresheaves on the Grothendieck construction of the fixed base. We also prove that this equivalence is pseudonatural in the base and that it restricts to discrete opfibrations with small fibres and copresheaves. Our result is a 2-dimensional generalization of the equivalence between slices of copresheaves and copresheaves on slices. We can think of the indexed Grothendieck construction as a simultaneous Grothendieck construction on every index that takes into account all bonds between different indexes.
\end{abstract}

\date{July 29, 2023}
\maketitle
				
\setcounter{tocdepth}{1}
\tableofcontents

\section{Introduction}

The Grothendieck construction, also known as the construction of the category of elements, is a fundamental tool in category theory. It establishes an equivalence between indexed categories and Grothendieck fibrations. The construction reorganizes the data of an indexed family of categories in a total category equipped with a functor that tells which index each object came from. It also abstractly captures the concept of change of base.

Originally introduced by Grothendieck in~\cite{grothendieck_revetsetalegroupefond} in a purely geometrical context, the Grothendieck construction has since found numerous applications in both logic and algebra. In algebra, it brought for example to the total category of all modules, collecting together $R$-modules for every ring $R$. While logicians have particularly used the restriction of the Grothendieck construction to the equivalence between families of sets indexed over a category and discrete fibrations (with small fibres). This allowed to consider such families of sets without mentioning morphisms that land into the universe of all sets. We believe that the general framework of the Grothendieck construction could yield new applications in logic, especially within the realm of 2-dimensional logic. For this purpose, the 2-category of elements, a natural extension of the Grothendieck construction, is particularly promising. This extension  has been studied in detail by the second author in~\cite{mesiti_twosetenrgrothconstrlaxnlimits}.

The equivalence between indexed categories and Grothendieck fibrations offers the advantages of both worlds. Additionally, the Grothendieck construction itself has significant and useful implications. Notably, it allows to conicalize all weighted $\Set$-enriched (i.e.\ ordinary) limits, and to almost conicalize weighted $2$-limits (see Street's~\cite{street_limitsindexedbycatvalued} and the second author's~\cite{mesiti_twosetenrgrothconstrlaxnlimits}). This also presents every presheaf as a colimit of representables and gives the famous explicit formula for the ordinary Kan extension.

In this paper, we produce an indexed version of the Grothendieck construction, that does not seem to appear in the literature. Our main results (\thex\ref{teormain} and \thex\ref{teorrestrictstodisc}) are condensed in the following theorem. Op and non-split variations are also considered in \remx\ref{remopvariations} and \remx\ref{remtwocatvariations}.

\begin{teor}\label{teornellintro}
	Let $\A$ be a small category and consider the functor $2$-category $\m{\A}{\CAT}$. For every $2$-functor $F\:\A\to \CAT$, there is an equivalence of categories
	$$\Fib[o][\m{\A}{\CAT}][F]\simeq \m{\Grothdiag{F}}{\CAT}$$
	between split opfibrations in the $2$-category $\m{\A}{\CAT}$ over $F$ and $2$-\pteor{co}presheaves on the Grothendieck construction $\Groth{F}$ of $F$.
	
	This restricts to an equivalence of categories
	$$\Fib+[b][\m{\A}{\CAT}][F]\simeq \m{\Grothdiag{F}}{\Set}$$
	between discrete opfibrations in $\m{\A}{\CAT}$ over $F$ with small fibres and $1$-copresheaves on $\Groth{F}$.
	
	Moreover, both the equivalences of categories above are pseudonatural in $F$.
\end{teor}

When $\A=\1$, we recover the usual Grothendieck construction. Indeed $\m{\A}{\CAT}$ reduces to $\CAT$, a $2$-functor $F\:\1\to \CAT$ is just a small category $\C$ and $\Groth{F}=\C$. So we find
$$\Fib[o][n][\C]\simeq \m{\C}{\CAT}.$$
But we introduce an indexed version of the Grothendieck construction that allows $\A$ to be an arbitrary small category and $F\:\A\to \CAT$ to be an arbitrary $2$-functor. Interestingly, the data of such general opfibrations in $\m{\A}{\CAT}$ over $F$ are still packed in a $\CAT$-valued copresheaf, now on the Grothendieck construction of $F$.

We can think of the indexed Grothendieck construction as a simultaneous Grothendieck construction on every index $A\in \A$, taking into account the bonds between different indexes. Indeed, an opfibration $\phi$ in $\m{\A}{\CAT}$ is, in particular, a natural transformation such that every component $\phi_A$ is a Grothendieck opfibration (in $\CAT$). Our construction essentially applies the quasi-inverse of the usual Grothendieck construction to every component of $\phi$ at the same time. All the obtained copresheaves in $\CAT$ are then collected into a single total copresheaf, exploiting the usual Grothendieck construction on $F$.

The restricted equivalence of categories
$$\Fib+[b][\m{\A}{\CAT}][F]\simeq \m{\Grothdiag{F}}{\Set}$$
further reduces, when $F\:\A\to \Set$, to the well known equivalence
$$\slice{\m{\A}{\Set}}{F}\simeq \m{\Grothdiag{F}}{\Set}.$$
When $F$ is a representable $\y{A}\:\A\to \Set$, this is the famous equivalence
$$\slice{\m{\A}{\Set}}{\y{A}}\simeq \m{\slice{\A}{A}}{\Set}$$
between slices of (co)presheaves and (co)presheaves on slices. Our theorem also guarantees its pseudonaturality in $A$, which does not seem to be stated in the literature.

The last equivalence between slices of (co)presheaves and (co)presheaves on slices had many applications in geometry and logic. In particular, it is the archetypal case of the fundamental theorem of elementary topos theory, showing that every slice of a Grothendieck topos is a Grothendieck topos. Our equivalence
$$\Fib[o][\m{\A}{\CAT}][F]\simeq \m{\Grothdiag{F}}{\CAT}$$
gives a $2$-dimensional generalization of this, and we thus expect it to be very fruitful. Indeed, the concept of (op)fibrational slice has recently been proposed as the correct upgrade of slices to dimension 2. Rather than taking all maps into a fixed element, we restrict to the (op)fibrations over that element. This idea appears in Ahrens, North and van der Weide's~\cite{ahrensnorthvanderweide_bicategoricaltypetheory}, where it is attributed to Shulman. Our equivalence can be thought of as saying that every (op)fibrational slice of a Grothendieck $2$-topos is again a Grothendieck $2$-topos.

Our motivating application of the indexed Grothendieck construction is to produce a nice candidate for a 2-classifier in the 2-category of 2-presheaves, in line with Hofmann and Streicher's~\cite{hofmannstreicher_liftinggrothuniverses}. We describe this in \exax\ref{exhofstreichuniv}. The second author has shown in the following paper~\cite{mesiti_twoclassifiersdensegenstacks} that such candidate is indeed a 2-dimensional classifier in $\m{\A\op}{\CAT}$, towards a $2$-dimensional elementary topos structure on $\m{\A\op}{\CAT}$. A 2-classifier, which is a generalization of the concept of subobject classifier to dimension 2, proposed by Weber in~\cite{weber_yonfromtwotop}, can also be thought of as a Grothendieck construction inside a 2-category. So it is natural to expect an indexed version of the Grothendieck construction to give a 2-classifier in the 2-category of 2-presheaves.

The strategy to prove our main theorem will be to use that the Grothendieck construction of a $2$-functor $F\:\A\to \CAT$ is equivalently the oplax colimit of $F$. So that we will be able to apply the usual Grothendieck construction on every index $A\in \A$. We will then show that all the opfibrations produced for each $A$ can be collected into an opfibration in $\m{\A}{\CAT}$ over $F$. For this, we will also need to prove that the usual Grothendieck construction is pseudonatural in the base category. The chain of abstract processes above will then be very useful to conclude the pseudonaturality in $F$ of the indexed Grothendieck construction. 

We will also give an explicit description of the indexed Grothendieck construction, in \conx\ref{consindexedgroth}, so that it can be applied more easily.

We will conclude showing some interesting examples, choosing particular $\A$'s and $F$'s in \thex\ref{teornellintro}. Among them, we will consider the cases $\A=\2$ (arrows between opfibrations) and $\A=\Delta$ (cosimplicial categories).

\subsection*{Outline of the paper}

In Section~\ref{sectionpropertiesgrothconstr}, we recall that the Grothendieck construction can be equivalently expressed as an oplax colimit and as a lax comma object. We prove that the equivalence of categories given by the Grothendieck construction is pseudonatural in the base category.

In Section~\ref{sectionopfibintwopresh}, after recalling the notion of opfibration in a $2$-category, we show an equivalent characterization of the opfibrations in $\m{\A}{\CAT}$. This also allows us to define {having small fibres} for a discrete opfibration in $\m{\A}{\Cat}$. We produce a pseudofunctor $F\mto \Fib[o][\m{\A}{\CAT}][F]$.\v[0.2]

In Section~\ref{sectionindexedgrothconstr}, we present our main theorems, proving an equivalence of categories between (split) opfibrations in $\m{\A}{\CAT}$ over $F$ and $2$-copresheaves on $\Groth{F}$. We also show that such equivalence is pseudonatural in $F$. We present the explicit indexed Grothendieck construction.

In Section~\ref{sectionexamples}, we show some interesting examples. In particular, we obtain a nice candidate for a Hofmann--Streicher universe in 2-presheaves.

\subsection*{Notation}

We fix Grothendieck universes $\U$, $\V$ and $\W$ such that $\U\in \V\in \W$. We denote as $\Set$ the category of $\U$-small sets, as $\Cat$ the 2-category of $\V$-small categories (i.e.\ categories such that both the collections of their objects and of their morphisms are $\V$-small) and as $\CATlarge$ the 2-category of $\W$-small categories. So that $\Set\in \Cat$ and the underlying category $\Cat_0$ of $\Cat$ is in $\CATlarge$. \predfn{Small category} will mean $\V$-small category. \predfn{Small fibres}, for a discrete opfibration in $\Cat$, will mean $\U$-small fibres. \predfn{2-category} will mean a $\W$-small $\Cat$-enriched category. \predfn{Small 2-category} will mean $\V$-small 2-category.

\section{Some properties of the Grothendieck construction}\label{sectionpropertiesgrothconstr}

In this section, we recall two useful equivalent characterizations of the usual Grothendieck construction. We then prove that the equivalence of categories given by the Grothendieck construction is pseudonatural in the base category. 

It is known that the Grothendieck construction $\Groth{F}$ of a 2-functor $F\:\C\to \CAT$ with $\C$ a category is equivalently the oplax colimit of $F$. As we could not find a proof of this in the literature, we show a proof below (\thex\ref{teorgrothisoplaxcolim}). Another useful characterization is that $\Groth{F}$ is equivalently the lax comma object from $\1\:\1\to \CAT$ to $F$ in $\twoCATlax$. This has been explored in detail by the second author in~\cite{mesiti_twosetenrgrothconstrlaxnlimits} and is recalled below (\thex\ref{teorgrothconstrislaxcomma}).

\begin{rem}
	In this paper, we focus on the Grothendieck construction of (strict) $2$-functors $F\:\C\to \CAT$ with $\C$ a small category, which correspond with split opfibrations over $\C$. We will consider variations of this setting in \remx\ref{remopvariations} (fibrations) and \remx\ref{remtwocatvariations} ($\C$ a $2$-category and non-split opfibrations).
\end{rem}

We first recall some basic definitions (\defx\ref{recgrothopfib} and \conx\ref{recgrothconstr}) from Jacobs's book~\cite{jacobs_catlogicandtypetheory}. Such concepts have been introduced by Grothendieck in~\cite{grothendieck_revetsetalegroupefond}.

\begin{defne}[\cite{grothendieck_revetsetalegroupefond}]\label{recgrothopfib}
	A functor $p\:\E\to \C$ is called a \dfn{\pteor{Grothendieck} opfibration} (in $\CAT$) (over $\C$) if for every object $E\in \E$ and every morphism $f\:p(E)\to C$ in $\C$, there exists an opcartesian lifting $\ov{f}^E\:E\to f\stb E$ of $f$ to $E$
	\fibsq{E}{f\stb E}{p(E)}{C}{\ov{f}^E}{p}{f}
	\dfn{Opcartesian} means that for every $E'\in \E$, every morphism $w\:C\to p(E')$ in $\C$ and every morphism $e\:E\to E'$ in $\E$ such that $p(e)=w\c f$, there exists a unique morphism $v\:f\stb E\to E'$ such that $p(v)=w$ and $v\c \ov{f}^E = e$.
	\fibsquniv{E}{f\stb E}{p(E)}{C}{\ov{f}^E}{p}{f}{E'}{p(E')}{e}{\exists ! \h v}{w}
	We call \dfn{cleavage} a choice of cartesian liftings $\ov{f}^E$ for every $f$ and $E$. An opfibration with a cleavage is called \dfn{split} if the cleavage is functorial in the following sense:
	\begin{enum}
		\item for every $E\in \E$ we have $\ov{\id{}}^E=\id{E}$;
		\item for every $E\in \E$ and morphisms $f\:p(E)\to C$ and $g\:C\to C'$ in $\C$ we have
		$$\ov{g}^{f\stb E}\c\ov{f}^E=\ov{(g\c f)}^E.$$
	\end{enum}
	 A \dfn{cleavage preserving morphism} between split opfibrations $p\:\E\to \C$ and $q\:\F\to \D$ is a commutative square in $\CAT$
	 \sq[n][5.5][6.5]{\E}{\F}{\C}{\D}{H}{p}{q}{K}
	 such that for every $E\in \E$ and every $f\:p(E)\to C$ in $\C$ we have $H(\ov{f^E})=\ov{K(f)}^{H(E)}$.
	 \begin{eqD*}
	 	\fibsqN{H(E)}{H(f\stb E)}{q(H(E))}{q(H(f\stb E))}{H(\ov{f}^E)}{q}{q(H(\ov{f}^E))}\quad \quad \quad
	 	\fibsqN{H(E)}{K(f)\stb H(E)}{K(p(E))}{K(C)}{\ov{K(f)}^{H(E)}}{q}{K(f)}
	 \end{eqD*}
	 If we restrict the attention to split opfibrations over a fixed $\C$, we require cleavage preserving morphisms $(H,K)$ to have $K=\Id{\C}$.
	 
	Split opfibrations over $\C$ and cleavage preserving morphisms form a category $\Fib[o][n][\C]$.
	
	A functor $p\:\E\to \C$ is called a \dfn{discrete opfibration} (in \CAT) (over \C) if for every object $E\in \E$ and every morphism $f\:p(E)\to C$ in $\C$, there exists a unique lifting $\ov{f}^E\:E\to f\stb E$ of $f$ to $E$. In particular, $\ov{f}^E$ is cartesian and $p$ is a Grothendieck opfibration. Cleavage preserving morphisms are just commutative squares in $\CAT$ and we obtain a category $\Fib[b][n][\C]$ of discrete opfibrations over $\C$. We denote as $\Fib+[b][n][\C]$ its full subcategory on the discrete opfibrations with small fibres.
\end{defne}

\begin{rem}
	The pullback $H\st p$ of a split opfibration $p\:\E\to \C$ along $H\:\D\to \C$ is a split opfibration. We can choose the cleavage of $H\st p$ to make the universal square that exhibits the pullback into a cleavage preserving morphism.
\end{rem}

\begin{cons}[\cite{grothendieck_revetsetalegroupefond}]\label{recgrothconstr}
	Let $\C$ be a small category and let $F\:\C\to \CAT$ be a $2$-functor. We can think of $F$ as a family of categories indexed over a category, taking into account any bond between different indexes. The Grothendieck construction is a process of reorganization of these data in terms of a single total category equipped with a projection functor that tells which index each object came from. The categories of the family are then recovered by taking the fibres of this projection functor. This process is based on the idea of taking the disjoint union of the categories of the family, but it also applies a change of scalars operation to handle the bonds between different indexes.
	
	The \dfn{Grothendieck construction} of $F$ is the functor $\groth{F}\:\Groth{F}\to \C$ of projection on the first component from the category $\Groth{F}$, which is defined as follows:
		\begin{description}
			\item[an object of $\Groth{F}$] is a pair $(C,X)$ with $C\in \C$ and $X\in F(C)$;
			\item[a morphism $(C,X)\to (D,X')$ in $\Groth{F}$] is a pair $(f,\alpha)$ with $f\:C\to D$ a morphism in $\C$ and $\alpha\:F(f)(X)\to X'$ a morphism in $F(D)$.
		\end{description}
	$\groth{F}\:\Groth{F}\to \C$ is a split opfibration, with cleavage given by the morphisms $(f,\id{})$.
\end{cons}

\begin{rem}\label{remfactorizationcartvert}
	Notice that every morphism $(f,\alpha)\:(C,X)\to (D,X')$ in $\Groth{F}$ can be factorized as
	$$(C,X)\ar{(f,\id{})}(D,F(f)(X))\ar{(\id{},\alpha)}(D,X')$$
	That is, as a cartesian morphism of the cleavage followed by a morphism which is over the identity (also called vertical morphism).
	
	This means that the Grothendieck construction of $F$ is somehow given by collecting all $F(C)$ together, where we have the morphisms $(\id{},\alpha)$, and adding the morphisms $(f,\id{})$ to handle change of index. This idea will be made precise in \thex\ref{teorgrothisoplaxcolim}.
\end{rem}

The following fundamental theorem is due to Grothendieck~\cite{grothendieck_revetsetalegroupefond} (see also Borceux's~\cite{borceux_handbookcatalgtwo}). 

\begin{teor}[\cite{grothendieck_revetsetalegroupefond}]\label{usualgrothconstr}
	The Grothendieck construction extends to an equivalence of categories
	$$\groth{-}\:\m{\C}{\CAT}\aequi \Fib[o][n][\C]$$
	Given a natural transformation $\gamma\:F\aR{}G\:\C\to \CAT$, the functor $\groth{\gamma}\:\Groth{F}\to \Groth{G}$ is defined to send $(C,X)$ to $(C,\gamma(X))$ and $(f,\alpha)\:(C,X)\to (D,X')$ to $(f,\gamma_D(\alpha))$.
	
	The quasi-inverse is given by taking fibres on every $C\in \C$.
	
	Moreover, the equivalence above restricts to an equivalence of categories
	$$\groth{-}\:\m{\C}{\Set}\aequi \Fib+[b][n][\C]$$
\end{teor}

Aiming at proving that the Grothendieck construction is equivalently given by an oplax colimit, we recall the definition of oplax colimit.

\begin{defne}\label{defoplaxcolim}
	Let $F\:\C\to \D$ be a $2$-functor with $\C$ small. The \dfn{oplax \pteor{conical} colimit of $F$}, denoted as $\oplaxcolim{F}$, is (if it exists) an object $K\in \D$ together with an isomorphism of categories
	$$\HomC{\D}{K}{U}\iso \HomC{\moplax{\C\op}{\CAT}}{\Delta 1}{\HomC{\D}{F(-)}{U}}$$
	$2$-natural in $U\in \D$, where $\moplax{\C\op}{\CAT}$ is the $2$-category of $2$-functors, oplax natural transformations and modifications from $\C\op$ to $\CAT$. $\Delta 1$ is the functor which is constant at singleton category $\1$ and the right hand side of the isomorphism above should be thought as the category of oplax cocones on $F$ with vertex $U$. Indeed we also have an isomorphism
	$$\HomC{\moplax{\C\op}{\CAT}}{\Delta 1}{\HomC{\D}{F(-)}{U}}\iso \HomC{\mlax{\C}{\D}}{F}{\Delta U}$$
	$2$-natural in $U$, where $\mlax{\C}{\D}$\v[0.5] is the $2$-category of $2$-functors, lax natural transformations and modifications, and $\Delta U$ is the functor which is constant at $U$.
\end{defne}

\begin{rem}\label{remoplaxcolimits}
	When $\oplaxcolim{F}$ exists, taking $U=K$ and considering the identity on $K$ gives us in particular a lax natural transformation
	$$\lambda\:F\alax{}\Delta K$$
	which is called the \dfn{universal oplax cocone} on $F$.
	
	An equivalent way to show that $K=\oplaxcolim{F}$ is to exhibit such a lax natural transformation $\lambda$ that is universal in the following $2$-categorical sense:
	\begin{enum}
		\item for every lax natural transformation $\sigma\:F\alax{}\Delta U$, there exists a unique\v[-0.2] morphism $s\:K\to U$ in $\D$ such that $\Delta s \c \lambda=\sigma$;
		\item for every $s,t\:K\to U$ in $\D$ and every modification $\Xi\:\Delta s\c \lambda\aM{} \Delta t\c \lambda$, there exists a unique $2$-cell $\chi\:s\aR{}t$ in $\D$ such that $\Delta \chi\star \lambda=\Xi$.
	\end{enum}
\end{rem}

We will need the following known characterization of the Grothendieck construction. As we could not find a proof of this in the literature, we show a proof here.

\begin{teor}\label{teorgrothisoplaxcolim}
	Let $\C$ be a small category and let $F\:\C\to \CAT$ be a $2$-functor. The Grothendieck construction $\Groth{F}$ of $F$ is equivalently the oplax (conical) colimit of the $2$-diagram $F$:
	$$\Grothdiag{F}=\oplaxcolim{F}$$
\end{teor} 
\begin{proof}
	Following \remx\ref{remoplaxcolimits}, we produce a lax natural transformation $\opn{inc}\:F\alax{}\Delta \Groth{F}$ and prove that it is universal in the $2$-categorical sense. For every $C\in \C$ we define the component of $\opn{inc}$ on $C$ to be the functor
	\begin{fun}
		\opn{inc}_{\h[1]C} & \: & F(C) & \too & \Groth{F} \\[0.65ex]
		&& \fibdiag{X}{\alpha}{X'} & \mto & \h[-3.5]\fibdiag{(C,X)}{(\id{},\alpha)}{(C,X')}
	\end{fun}
	For every morphism $f\:C\to D$ in $\C$, we define the structure $2$-cell of $\opn{inc}$ on $f$ to be the natural transformation
	\begin{cd}[5][7]
		F(C) \arrow[d,"{F(f)}"'] \arrow[r,"{\opn{inc}_{\h[1]C}}"]\arrow[d,Rightarrow,shift left=5.3ex,"{\opn{inc}_f}"{pos=0.43},shorten >=0.6ex]\& \Groth{F}\\
		F(D)\arrow[ru,bend right,"{\opn{inc}_{\h[1]D}}"'{inner sep=0.35ex}]
	\end{cd}
	that has components ${(\opn{inc}_f)}_X=(f,\id{})$ for every $X\in F(C)$. The naturality of $\opn{inc}_f$ expresses 
	$$(f,\id{})\c (\id{},\alpha)=(\id{},F(f)(\alpha))\c (f,\id{}).$$
	As explained with more detail in the second author's~\cite{mesiti_twosetenrgrothconstrlaxnlimits}, to get the whole $\Groth{F}$ we just need the two kinds of morphisms $(\id{},\alpha)$ and $(f,\id{})$ as building blocks. This is what will ensure the universality of $\opn{inc}$. Composition of morphisms of type $(\id{},\alpha)$ corresponds with the functoriality of $\opn{inc}_{\h[1]C}$. While composition of morphisms of type $(f,\id{})$ corresponds with the lax naturality of $\opn{inc}$. We could then define general morphisms to be formal composites $(\id{},\alpha)\c (f,\id{})$, following the factorization of morphisms in $\Groth{F}$ described in \remx\ref{remfactorizationcartvert}. And the equation above, that swaps the two kinds of morphisms, tells how to reduce every composition to this form.
	
	We prove that $\opn{inc}$ is universal in the $2$-categorical sense. Given a lax natural transformation $\sigma\:F\alax{}\Delta U$, we show that there exists a unique $s\:\Groth{F}\to U$ such that $\Delta s\c \opn{inc}=\sigma$. These conditions impose to define for every $(f,\alpha)\:(C,X)\to (D,X')$ in $\Groth{F}$
	$$s(C,X)=\left(s\c \opn{inc}_{\h[1]C}\right)(X)=\sigma_C(X)$$
	$$s(\id{D},\alpha)=\left(s\c \opn{inc}_{\h[1]D}\right)(\alpha)=\sigma_D(\alpha)$$
	$$s(f,\id{})=s\left({(\opn{inc}_f)}_X\right)={(\sigma_f)}_X$$
	So by the factorization described in \remx\ref{remfactorizationcartvert}, we need to define
	$$s(f,\alpha)=s(\id{},\alpha)\c s(f,\id{})=\sigma_D(\alpha)\c {(\sigma_f)}_X.$$
	$s$ is a functor by naturality of $\sigma_g$, functoriality of $\sigma_E$ and lax naturality of $\sigma$. And $\Delta s\c \opn{inc}= \sigma$ by construction.
	
	Take now $s,t\:\Groth{F}\to U$ and a modification $\Xi\:\Delta s \c \opn{inc}\aM{}\Delta t\c \opn{inc}$. $\Xi$ has as components on $C$ natural transformations $\Xi_C\:s\c \opn{inc}_{\h[1]C}\aR{}t\c \opn{inc}_{\h[1]C}$. We show that there exists a unique natural transformation $\chi\:s\aR{}t$ such that $\Delta \chi\star \opn{inc}=\Xi$. We need to define
	$$\chi_{(C,X)}={\left(\chi\star \opn{inc}_{\h[1]C}\right)}_X=\Xi_{C,X}$$
	and this works. So $\opn{inc}$ is universal.
\end{proof}

\begin{rem}\label{remoplaxcoliminCAT}
	Let $\C$ be a small category and let $F\:\C\to \Cat$ be a 2-functor. $\Groth{F}$ is also the oplax (conical) colimit, with respect to the enrichment over $\CATlarge$, of the 2-diagram $\C\ar{F}\Cat \ito \CATlarge$. Indeed the Grothendieck construction of the latter composite is clearly just $\Groth{F}$.
\end{rem}

We also need the following result from the second author's~\cite{mesiti_twosetenrgrothconstrlaxnlimits}.

\begin{teor}[\cite{mesiti_twosetenrgrothconstrlaxnlimits}]\label{teorgrothconstrislaxcomma}
	Let $\C$ be a category and $F\:\C\to \CAT$ be a $2$-functor. The Grothendieck construction $\Groth{F}$ of $F$ is equivalently given by the lax comma object
	\begin{eqD*}
		\sq*[l][6][6][\h[-1.9]\lax \opn{comma}][2.7][2.2][0.55]{\Groth{F}}{\1}{\C}{\CAT}{}{\groth{F}}{\1}{F}
	\end{eqD*}
	in $\twoCATlax$ \pteor{the lax $3$-category of $2$-categories, $2$-functors, lax natural transformations and modifications}.
	
	As a consequence, it is then also given by the strict $3$-pullback in $\twoCATlax$ between $F$ and the replacement $\tau$ of $\1\:\1\to \CAT$ obtained by taking the lax comma object of $\1\:\1\to \CAT$ along the identity of $\CAT$ \pteor{that is a lax $3$-dimensional version of the lax limit of the arrow $\1\:\1\to \CAT$}:
	\begin{eqD*}
		\begin{cd}*[5.85][5]
			\Groth{F}\PB{rd} \arrow[d,"\groth{F}"'] \arrow[r] \& \CAT\bl \arrow[d,"{\tau}"'] \arrow[r]\& \1 \arrow[d,"{\1}"]\arrow[ld,Rightarrow,shorten <=2.7ex,shorten >=2.2ex,"{\h[-2.3]\lax \opn{comma}}"{pos=0.57}] \\
			\B\arrow[r,"F"'] \& \CAT \arrow[r,equal]\& \CAT
		\end{cd}
	\end{eqD*}
	The domain of $\tau$ is a lax pointed version of $\CAT$, whence the notation $\CAT\bl$.
\end{teor}

We can now prove that the Grothendieck construction is pseudonatural in the base category. Such result does not seem to appear in the literature.

\begin{prop}\label{propgrothconstrispseudonat}
	The equivalence of categories $$\G{\C}\:\m{\C}{\CAT}\aequi \Fib[o][n][\C]$$
	of \thex\ref{usualgrothconstr} given by the Grothendieck construction is pseudonatural in $\C\in \CAT\op$.
\end{prop}
\begin{proof}
	The assignment $\C\mto \Fib[o][n][\C]$ extends to a pseudofunctor
	$$\Fib[o][n][-]\:\CAT\op\to \CATlarge$$
	that on the underlying category of the domain $\CAT\op$ is a restriction of the pseudofunctor that does the pullback. So given $H\:\D\to \C$ and a split opfibration $p\:\E\to \C$, we define the action of $\Fib[o][n][-]$ on $H$ to be the pullback functor $H\st$.
	\sq[n][5.5][6.5]{H\st \E}{\E}{\D}{\C}{\t{H}}{H\st p}{p}{H}
	Given a natural transformation $\alpha\:H\aR{}K\:\D\to \C$, we use the cleavage of $p$ to define $\Fib[o][n][\alpha]=\alpha\st$ as the natural transformation that has as component on $p$ the functor
	$$\alpha\st p\:H\st \E\to K\st \E$$
	that sends $(D,E)\in H\st \E$ to $(D,(\alpha_D)\stb E)\in K\st \E$. We will prove that $\Fib[o][n][-]$ is indeed a pseudofunctor in \prox\ref{propopfib-}, for general split opfibrations in a $2$-category. In that general setting, we can define $\alpha\st$ by lifting a $2$-cell along an opfibration. This point of view is helpful to apply below the universal property of the lax comma object, using \thex\ref{teorgrothconstrislaxcomma}.
	
	We define a pseudonatural transformation
	$$\G{-}\:\m{-}{\CAT}\apseudo{}\Fib[o][n][-]$$
	that has component on $\C$ given by $\G{C}$. Given a functor $H\:\D\to \C$, we define the structure $2$-cell $\G{H}$ to be the natural isomorphism
	\begin{cd}[5.85][5.85]
		\m{\C}{\CAT}\arrow[r,simeq,"{}"] \arrow[d,"{-\c H}"']\& \Fib[o][n][\C]\arrow[d,"{H\st}"] \\
		\m{\D}{\CAT}\arrow[ru,iso,"{\G{H}}"]\arrow[r,simeq,"{}"] \& \Fib[o][n][\D]
	\end{cd}
	that is given by the pseudofunctoriality of the pullback (or actually by the fact that the pullback of a lax comma is isomorphic to the lax comma with the composite), thanks to \thex\ref{teorgrothconstrislaxcomma}:
	\begin{cd}[5.5][5.5]
		\Groth{(F\c H)}\arrow[rd,iso,"{{\left(\G{H}\right)}_F^{-1}}"]\arrow[rrrd,bend left=18,"{}"]\arrow[rdd,bend right,"{\groth{F\c H}}"']\&[-4.5ex]\& \&[-0.6ex]\\[-4.5ex]
		\& H\st \Groth{F}\arrow[d,"{H\st \groth{F}}"'] \arrow[r,"{}"]\& \Groth{F}\arrow[d,"{\groth{F}}"']\arrow[r,"{}"] \& \CAT\bl \arrow[d,"{\tau}"] \\
		\& \D\arrow[r,"{H}"'] \& \C\arrow[r,"{F}"'] \& \CAT
	\end{cd}
	$\G{H}$ is indeed a natural transformation thanks to the universal property of the lax comma object. And $\G{-}$ satisfies the $1$-dimensional condition of pseudonatural transformation by the pseudofunctoriality of the pullback (choosing the pullbacks along identities to be the identity).
	
	Take now a natural transformation $\alpha\:H\aR{}K\:\D\to \C$. In order to prove the $2$-dimensional condition of pseudonatural transformation for $\G{-}$, we need to show that the following square is commutative for every $F\:\C\to \CAT$:
	\sq[n][5.3][5.8]{H\st \Groth{F}}{K\st\Groth{F}}{\Groth{(F\c H)}}{\Groth{(F\c K)}}{\alpha\st_{\groth{F}}}{{\left(\G{H}\right)}_F}{{\left(\G{K}\right)}_F}{\groth{F\star \alpha}}
	This is shown by the universal property of the lax comma object (or of the pullback) $\Groth{F}$. For this we use the fact that the chosen cleavage on $\groth{F}\:\Groth{F}\to \C$ (with $\ov{f}^{(C,X)}=(f,\id{})$, see \conx\ref{recgrothconstr}) makes the square
	\sq[p][5.5][5.5]{\Groth{F}}{\CAT\bl}{\C}{\CAT}{}{\groth{F}}{\tau}{F}
	into a cleavage preserving morphism.
\end{proof}

\section{Opfibrations in the 2-category of 2-presheaves}\label{sectionopfibintwopresh}

In this section, after recalling the notion of opfibration in a $2$-category, we characterize the opfibrations in the functor $2$-category $\m{\A}{\CAT}$ with $\A$ a small category. We will show that such characterization restricts to one of discrete opfibrations as well. This will allow us to define \predfn{having small fibres} for a discrete opfibration in $\m{\A}{\CAT}$ (\defx\ref{defhavingsmallfibres}).

The definition of (op)fibration in a $2$-category is due to Street~\cite{street_fibandyonlemma}, in terms of algebras for a $2$-monad. It is known that we can equivalently define (op)fibrations in a $2$-category by representability, as done in Weber's~\cite{weber_yonfromtwotop}.

\begin{defne}\label{defopfibinatwocat}
	Let $\L$ be a $2$-category. A \dfn{split opfibration in $\L$} is a morphism $\phi\:G\to F$ in $\L$ such that for every $X\in \L$ the functor 
	$$\phi\c -\:\HomC{\L}{X}{G}\to\HomC{\L}{X}{F}$$ 
	induced by $\phi$ between the hom-categories is a split Grothendieck opfibration (in $\CAT$) and for every morphism $\lambda\:K\to X$ in $\L$ the commutative square
	\sq[5.85][4.7]{\HomC{\L}{X}{G}}{\HomC{\L}{K}{G}}{\HomC{\L}{X}{F}}{\HomC{\L}{K}{F}}{-\c\lambda}{\phi\c -}{\phi\c -}{-\c \lambda}
	is cleavage preserving.
	
	We call $\phi$ a \dfn{discrete opfibration in $\L$} if for every $X$ the functor $\phi\c -$ above is a discrete opfibration (in $\CAT$). In this case, the second condition is automatic.
	
	Given split opfibrations $\phi\:G\to F$ and $\psi\:H\to F$ in $\L$ over $F$, a cleavage preserving morphism from $\phi$ to $\psi$ is a morphism $\xi\:\phi\to \psi$ in $\slice{\L}{F}$ such that for every $X\in \L$ the triangle
	\tr[-1][4]{\HomC{\L}{X}{G}}{\HomC{\L}{X}{H}}{ \HomC{\L}{X}{F}}{\xi\c -}{\phi\c -}{\psi\c -}
	is cleavage preserving.
	
	If $\phi$ and $\psi$ are discrete opfibrations, any morphism in $\slice{\L}{F}$ is cleavage preserving.
	
	Split opfibrations in $\L$ over $F$ and cleavage preserving morphisms form a category $\Fib[o][\L][F]$. We denote the full subcategory on discrete opfibrations in $\L$ as $\Fib[b][\L][F]$.
\end{defne}

\begin{rem}
	By definition, a (split) opfibration $\phi\:G\to F$ in $\L$ is required to lift every $2$-cell $\theta\:\phi\c \alpha\aR{}\beta$ to a cartesian $2$-cell $\ov{\theta}^\alpha\:\alpha\aR{}\theta\stb \alpha$. We can draw the following diagram to say that $\phi\c \theta\stb \alpha=\beta$ and $\phi\star\ov{\theta}^\alpha=\theta$.
		\begin{cd}[4.5][4.5]
			X \arrow[dd,equal]\arrow[rr,bend left=25,"{\alpha}",""'{name=A}]\arrow[rr,bend right=25,"{\theta\stb \alpha}"',""{name=B}]\& \& G \arrow[dd,"{\phi}"]\\[1ex]
			\& |[alias=W]| G \arrow[rd,bend left=20,"{\phi}"]\\[-5ex]
			X \arrow[ru,bend left=18,"{\alpha}"]\arrow[rr,bend right=18,"{\beta}"',""{name=Q}] \&\hphantom{.}\& F
			\arrow[from=A,to=B,Rightarrow,"{\h[2]\ov{\theta}^\alpha}"]
			\arrow[from=W,to=Q,Rightarrow,"{\h[2]\theta}"]
		\end{cd}
	$\ov{\theta}^\alpha$ cartesian means that for every $2$-cell $\rho\:\alpha\aR{}\alpha'$ and $2$-cell $\sigma\:\beta\aR{}\phi\c \alpha'$ such that $\phi\ast \rho=\sigma\c \theta$, there exists a unique $2$-cell $\nu\:\theta\stb\alpha\aR{}\alpha'$ such that $\phi\star \nu=\sigma$ and $\nu\c \ov{\theta}^\alpha=\rho$. Analogously, we can express being split in these terms.
	
	The second condition of \defx\ref{defopfibinatwocat} then requires the chosen lifting of $\theta\star \lambda$ to be $\ov{\theta}^\alpha\star \lambda$ (i.e.\ the chosen lifting of $\theta$ whiskered with $\lambda$).
	
	$\phi$ is a discrete opfibration in $\L$ when the liftings $\ov{\theta}^\alpha$ are unique.
	
	Cleavage preserving morphisms can be expressed analogously.
\end{rem}

\begin{rem}\label{rempullbackopfib}
	Pullbacks of split opfibrations are split opfibrations, because $\HomC{\L}{X}{-}$ preserves pullbacks (as it preserves all limits) and pullbacks of split opfibrations in $\CAT$ are split opfibrations in $\CAT$. We are also using (for the second condition) that we can choose the cleavage of the pullback of a split opfibration in $\L$ so that the universal square that exhibits the pullback is cleavage preserving.
\end{rem}

\begin{rem}
	We can of course apply \defx\ref{defopfibinatwocat} to $\L=\CAT$. The produced notion is equivalent to the usual notion of Grothendieck opfibration (of \defx\ref{recgrothopfib}). This is essentially because for $\L=\CAT$ it suffices to ask the above liftings for $X=\1$. We are then able to lift entire natural transformations $\theta$ as a consequence, componentwise. Analogously with discrete opfibrations in $\L=\CAT$.
	
	We extend this idea below and characterize opfibrations in $\m{\A}{\CAT}$.
\end{rem}

The following proposition does not seem to appear in the literature.

\begin{prop}\label{propcharactopfibacat}
	Let $\A$ be a small category and consider a morphism $\phi\:G\to F$ in $\m{\A}{\CAT}$ \pteor{i.e.\ a natural transformation}. The following facts are equivalent:
	\begin{enumT}
		\item $\phi\: G \to F$ is a split opfibration in $\m{\A}{\CAT}$;
		\item for every $A\in \A$ the component $\phi_A\:G(A) \to F(A)$ of $\phi$ on $A$ is a split opfibration \pteor{in $\CAT$} and for every morphism $h\:A\to B$ in $\A$ the naturality square 
		\sq[5.85][5.85]{G(A)}{G(B)}{F(A)}{F(B)}{G(h)}{\phi_A}{\phi_B}{F(h)}
		is cleavage preserving.
	\end{enumT}
	Analogously with discrete opfibrations in $\m{\A}{\CAT}$, where the condition on the naturality square of $(ii)$ is automatic.
\end{prop}
\begin{proof}
	We prove $(i)\aR{}(ii)$. Let $A\in \A$. Taking $X=\y{A}$ in \defx\ref{defopfibinatwocat} with $\L=\m{\A}{\CAT}$, we obtain that
	$$\phi\c -\:\HomC{\m{\A}{\CAT}}{\y{A}}{G}\to\HomC{\m{\A}{\CAT}}{\y{A}}{F}$$ 
	is a split opfibration in $\CAT$. By Yoneda lemma, we have isomorphisms that form a commutative square
	\begin{cd}[5.4][5]
		\HomC{\m{\A}{\CAT}}{\y{A}}{G}\arrow[r,iso,""]\arrow[d,"{\phi\c -}"']\& G(A) \arrow[d,"{\phi_A}"] \\
		{\HomC{\m{\A}{\CAT}}{\y{A}}{F}}\arrow[r,iso,""] \& F(A)
	\end{cd}
	We can then choose a cleavage on $\phi_A$ that makes it into a split opfibration in $\CAT$ such that the square above is cleavage preserving. Given $h\:A\to B$ in $\A$ we have that the naturality square of $\phi$ on $h$ is equal to the pasting
	\begin{cd}[5.5][5]
		G(A) \arrow[d,"{\phi_A}"']\arrow[r,iso,"{}"]\& \HomC{\m{\A}{\CAT}}{\y{A}}{G}\arrow[d,"{\phi\c -}"']\arrow[r,"{-\c \y{h}}"]\&
		\HomC{\m{\A}{\CAT}}{\y{B}}{G}\arrow[r,iso,""]\arrow[d,"{\phi\c -}"]\& G(B) \arrow[d,"{\phi_B}"] \\
		F(A) \arrow[r,iso,"{}"]\& \HomC{\m{\A}{\CAT}}{\y{A}}{F}\arrow[r,"{-\c \y{h}}"']\& {\HomC{\m{\A}{\CAT}}{\y{B}}{F}}\arrow[r,iso,""] \& F(B)
	\end{cd}
	and is thus cleavage preserving.
	
	When $\phi$ is a discrete opfibration, $\phi_A$ is discrete as well for every $A\in \A$.
	
	We now prove $(ii)\aR{}(i)$. Let $X\in \m{\A}{\CAT}$, $\alpha\:X\to G$, $\beta\:X\to F$ and consider $\theta\:\phi\c \alpha\aR{}\beta$. We need to produce a cartesian lifting $\ov{\theta}^\alpha\:\alpha\aR{}\theta\stb\alpha$ of $\theta$ to $\alpha$.
	\begin{cd}[4.3][4.5]
		X \arrow[dd,equal]\arrow[rr,bend left=25,"{\alpha}",""'{name=A}]\arrow[rr,bend right=25,"{\theta\stb \alpha}"',""{name=B}]\& \& G \arrow[dd,"{\phi}"]\\[1ex]
		\& |[alias=W]| G \arrow[rd,bend left=20,"{\phi}"]\\[-5ex]
		X \arrow[ru,bend left=18,"{\alpha}"]\arrow[rr,bend right=18,"{\beta}"',""{name=Q}] \&\hphantom{.}\& F
		\arrow[from=A,to=B,Rightarrow,"{\h[2]\ov{\theta}^\alpha}"]
		\arrow[from=W,to=Q,Rightarrow,"{\h[2]\theta}"]
	\end{cd}
	As $\theta\stb\alpha$ is a natural transformation and $\ov{\theta}^\alpha$ is a modification, we can define them on components. Given $A\in \A$ and $Z\in X(A)$, we define the image of the functor $(\theta\stb\alpha)_A$ on $Z$ and the morphism $(\ov{\theta}^\alpha)_{A,Z}$ in $G(A)$ to be given by the chosen cartesian lifting along $\phi_A$ of $\theta_{A,Z}$ to $\alpha_A(Z)$:
	\fibsq{\alpha_A(Z)}{(\theta\stb\alpha)_A(Z)}{\phi_A(\alpha_A(Z))}{\beta_A(Z)}{(\ov{\theta}^\alpha)_{A,Z}}{\phi_A}{\theta_{A,Z}}
	Given a morphism $f\:Z\to Z'$ in $X(A)$, we define $(\theta\stb\alpha)_A(f)$ by cartesianity of $(\ov{\theta}^\alpha)_{A,Z}$, making by construction $(\ov{\theta}^\alpha)_A$ into a natural transformation. $(\theta\stb\alpha)_A$ is then automatically a functor. In order to prove that $\theta\stb\alpha$ is a natural transformation, we need to show that for every $h\:A\to B$ in $\A$ the following square is commutative:
	\sq[5.5][5.5]{X(A)}{G(A)}{X(B)}{G(B)}{(\theta\stb)_A}{X(h)}{G(h)}{(\theta\stb)_B}
	This is straightforward using the hypothesis that $(G(h),F(h))$ is cleavage preserving. The argument shows at the same time that $\ov{\theta}^\alpha$ is a modification. $\ov{\theta}^\alpha$ is then a lifting of $\theta$ to $\alpha$ by construction, as this can be checked on components. It is straightforward to show that it is cartesian as well, inducing the required morphism on components by the cartesianity of all the $(\ov{\theta}^\alpha)_{A,Z}$. Coherences are shown using again that $(G(h),F(h))$ is cleavage preserving. $\phi$ is split because all $\phi_A$ are split.
	
	Given $\lambda\:K\to X$ in $\m{\A}{\CAT}$, we prove that
	\sq[5.85][4.7]{\HomC{\L}{X}{G}}{\HomC{\L}{K}{G}}{\HomC{\L}{X}{F}}{\HomC{\L}{K}{F}}{-\c \lambda}{\phi\c -}{\phi\c -}{-\c \lambda}
	is cleavage preserving. This means that
	\begin{cd}[4.3][4.5]
		K\arrow[r,"{\lambda}"]\arrow[dd,equal,"{}"]\& X \arrow[dd,equal]\arrow[rr,bend left=25,"{\alpha}",""'{name=A}]\arrow[rr,bend right=25,"{\theta\stb \alpha}"',""{name=B}]\& \& G \arrow[dd,"{\phi}"]\\[1ex]
		\&\& |[alias=W]| G \arrow[rd,bend left=20,"{\phi}"]\\[-5ex]
		K\arrow[r,"{\lambda}"]\&X \arrow[ru,bend left=18,"{\alpha}"]\arrow[rr,bend right=18,"{\beta}"',""{name=Q}] \&\hphantom{.}\& F
		\arrow[from=A,to=B,Rightarrow,"{\h[2]\ov{\theta}^\alpha}"]
		\arrow[from=W,to=Q,Rightarrow,"{\h[2]\theta}"]
	\end{cd}
	exhibits the chosen cartesian lifting of $\theta\star \lambda$ to $\alpha\c \lambda$. This works by construction, as the lifting of every $2$-cell along $\phi$ is reduced to lift morphisms of $F(A)$ along $\phi_A$ for every $A\in \A$.
	
	When $\phi_A$ is a discrete opfibration for every $A\in \A$, the argument above produces the needed cartesian liftings. We only need to show that such liftings are unique. But any lifting $\ov{\theta}^\alpha$ needs to have as component $(\ov{\theta}^\alpha)_{A,Z}$ on $\A\in \A$ and $Z\in X(A)$ the unique lifting of $\theta_{A,Z}$ to $\alpha_A(Z)$ along $\phi_A$.
\end{proof}

\begin{prop}\label{propcharactmoropfibacat}
	Let $\A$ be a small category and consider $\phi,\psi\in \Fib[o][\m{\A}{\CAT}][F]$. Let then $\xi\:\phi\to \psi$ be a morphism in $\slice{\m{\A}{\CAT}}{F}$. The following facts are equivalent:
	\begin{enumT}
		\item $\xi\:\phi\to \psi$ is cleavage preserving;
		\item for every $A\in \A$, the component $\xi_A\:\phi_A\to \psi_A$ is cleavage preserving \pteor{between split opfibrations in $\CAT$}.
	\end{enumT}
\end{prop}
\begin{proof}
	We prove $(i)\aR{}(ii)$. Given $A\in \A$ we have that
	\begin{cd}[5.5][5]
		G(A) \arrow[d,"{\phi_A}"']\arrow[r,iso,"{}"]\& \HomC{\m{\A}{\CAT}}{\y{A}}{G}\arrow[d,"{\phi\c -}"']\arrow[r,"{\xi\c-}"]\&
		\HomC{\m{\A}{\CAT}}{\y{A}}{H}\arrow[r,iso,""]\arrow[d,"{\psi\c -}"]\& H(A) \arrow[d,"{\psi_A}"] \\
		F(A) \arrow[r,iso,"{}"]\& \HomC{\m{\A}{\CAT}}{\y{A}}{F}\arrow[r,equal]\& {\HomC{\m{\A}{\CAT}}{\y{A}}{F}}\arrow[r,iso,""] \& F(A)
	\end{cd}
	is cleavage preserving.
	
	We prove $(ii)\aR{}(i)$. The equality of modifications that we need to prove can be checked on components, where it holds by hypothesis.
\end{proof}

Thanks to \prox\ref{propcharactopfibacat}, we can define \predfn{having small fibres} for a discrete opfibration in $\m{\A}{\Cat}$.

\begin{defne}\label{defhavingsmallfibres}
	Let $\A$ be a small category. A discrete opfibration $\phi\:G\to F$ in $\m{\A}{\Cat}$ \dfn{has small fibres} if for every $A\in \A$ the component $\phi_A$ of $\phi$ on $A$ has small fibres.
	
	We denote as $\Fib+[b][\m{\A}{\Cat}][F]$ the full subcategory of $\Fib[b][\m{\A}{\Cat}][F]$ on the discrete opfibrations with small fibres.
\end{defne}

\begin{rem}\label{remhavingsmallfibresispbstable}
	The property of having small fibres for a discrete opfibration in $\m{\A}{\Cat}$ is stable under pullbacks. Indeed taking components on $A\in \A$ preserves 2-limits in 2-presheaves and discrete opfibrations in $\Cat$ with small fibres are stable under pullbacks.
\end{rem}

We will also need the following result.

\begin{prop}\label{propopfib-}
	Let $\L$ be a $2$-category with pullbacks. The assignment $F\in \L\mto \Fib[o][\L][F]\in \CATlarge$ extends to a pseudofunctor 
	$$\Fib[o][\L][-]\:\L\op\to \CATlarge.$$
	
	Moreover, this pseudofunctor restricts to a pseudofunctor
	$$\Fib+[b][\L][-]\:\L\op\to \CATlarge.$$
\end{prop}
\begin{proof}
	On the underlying category of $\L\op$, we define $\Fib[o][\L][-]$ as the restriction of the pseudofunctor given by the pullback (to consider opfibrations rather than general morphisms). So given $\alpha\:F'\to F$ in $\L$, we have $$\Fib[o][\L][\alpha]=\alpha\st\:\Fib[o][\L][F]\to\Fib[o][\L][F']$$
	We are also using \remx\ref{rempullbackopfib}. We immediately get also the isomorphisms that regulate the image of identities and compositions. 
	
	Given a $2$-cell $\delta\:\alpha\aR{}\beta\:F'\to F$ in $\L$, we define $\Fib[o][\L][\delta]=\delta\st$ as the natural transformation with component on a split opfibration $\phi\:G\to F$ in $\L$ given by the morphism $\delta\st_\phi\:\alpha\st \phi\to \beta\st\phi$ induced by lifting $\delta\star \alpha\st\phi$ along $\phi$:
	\begin{cd}[6][7]
		\alpha\st G\arrow[rrd,bend left,"{}"'{name=Q}]\arrow[rdd,bend right,"{\alpha\st \phi}"'{inner sep=0.3ex}]\arrow[rd,dashed,"{\delta\st_\phi}"{inner sep=0.25ex}]\&[-5ex] \\[-5ex]
		\& \beta\st G\arrow[from=Q,Rightarrow,"{\ov{\delta\star \alpha\st\phi}}"{pos=0.35}]\arrow[r,"{}"]\arrow[d,"{\beta\st \phi}"']\& G\arrow[d,"{\phi}"] \\
		\& F' \arrow[r,bend left,"{\alpha}",""'{name=A}]\arrow[r,bend right,"{\beta}"',""{name=B}]\& F
		\arrow[from=A,to=B,Rightarrow,"{\delta}"]
	\end{cd}
	Indeed the codomain of the lifting of $\delta\star \alpha\st\phi$ along $\phi$ induces the morphism $\delta\st_\phi$ by the universal property of the pullback $\beta\st G$. $\delta\st$ is a natural transformation because the morphisms in $\Fib[o][\L][F]$ are cleavage preserving.
	
	$\Fib[o][\L][-]$ preserves identity $2$-cells and vertical compositions of $2$-cells because the objects of $\Fib[o][\L][F]$ are split. We already know that the isomorphisms that regulate the image of identities and compositions satisfy the $1$-dimensional coherences. It only remains to prove their naturality (actually, only the one for compositions). This essentially means that it preserves whiskerings, up to pasting with the isomorphisms that regulate the image of compositions. For whiskering on the left, this is true by the second condition of \defx\ref{defopfibinatwocat}. For whiskerings on the right, we use that the universal square that exhibits a pullback is cleavage preserving.
	
	Thus we conclude that $\Fib[o][\L][-]$ is a pseudofunctor. It then readily restricts to a pseudofunctor $\Fib+[b][\L][-]$.
\end{proof}

\section{Indexed Grothendieck construction}\label{sectionindexedgrothconstr}

In this section, we present our main results. We prove an equivalence of categories between split opfibrations in $\m{\A}{\CAT}$ over $F$ and $2$-copresheaves on $\Groth{F}$. This equivalence restricts to one between discrete opfibrations in $\m{\A}{\CAT}$ over $F$ with small fibres and $\Set$-valued copresheaves on $\Groth{F}$. We also show that both such equivalences are pseudonatural in $F$.

We introduce the explicit indexed Grothendieck construction and show how our results recover known useful results. In particular, we recover the equivalence between slices of presheaves over $F\:\A\to \Set$ and presheaves on $\Groth{F}$, that shows how the slice of a Grothendieck topos is a Grothendieck topos. We interpret our main theorem as a $2$-dimensional generalization of this.

Let $\A$ be a small category and consider the functor $2$-category $\m{\A}{\CAT}$.

\begin{rem}
	We aim at proving that for every $2$-functor $F\:\A\to \CAT$, there is an equivalence of categories
	$$\Fib[o][\m{\A}{\CAT}][F]\simeq \m{\Grothdiag{F}}{\CAT}$$
	between split opfibrations in $\m{\A}{\CAT}$ over $F$ (see \prox\ref{propcharactopfibacat}) and $2$-copresheaves on the Grothendieck construction $\Groth{F}$ of $F$.
	
	Our strategy will be to use \thex\ref{teorgrothisoplaxcolim}, that states that the Grothendieck construction $\Groth{F}$ of $F$ is equivalently the oplax colimit of the $2$-diagram $F\:\A\to \CAT$. Notice that a (strict) $2$-functor from a category to $\CAT$ is the same thing as a functor into the $1$-category $\Catzero$. In \remx\ref{remtwocatvariations}, we will say what we could do to extend our results to $\A$ a $2$-category or $F$ a pseudofunctor.
\end{rem}

\begin{prop}\label{propapplyoplaxcolim}
	There is an isomorphism of categories
	$$\m{\Grothdiag{F}}{\Catzero} \iso \HomC{\moplax{\A\op}{\CATlarge}}{\Delta 1}{\m{F(-)}{\Catzero}}$$
	which is \pteor{strictly} $2$-natural in $F$.
\end{prop}
\begin{proof}
	We obtain the isomorphism of categories in the statement by \thex\ref{teorgrothisoplaxcolim}, that proves that $\Groth{F}=\oplaxcolim{F}$ (see also \defx\ref{defoplaxcolim}). The isomorphism is $2$-natural in $F$ by a general result on weighted colimits, see Kelly's~\cite{kelly_basicconceptsofenriched} (Section~3.1). We can apply this result on an oplax colimit as well because by Street's~\cite{street_limitsindexedbycatvalued} any oplax colimit is also a weighted one.
\end{proof}

\begin{rem}
	Thanks to \prox\ref{propapplyoplaxcolim}, we can reduce ourselves to apply the usual Grothendieck construction on every index. For this we also need the pseudonaturality of the Grothendieck construction (\prox\ref{propgrothconstrispseudonat}).
\end{rem}

\begin{prop}\label{propgrothconstroneveryindex}
	There is an equivalence of categories
	$$\HomC{\moplax{\A\op}{\CATlarge}}{\Delta 1}{\m{F(-)}{\Catzero}}\simeq \HomC{\mPsoplax{\A\op}{\CATlarge}}{\Delta 1}{\Fib[o][\CAT][F(-)]}$$
	which is pseudonatural in $F$, where $\mPsoplax{\A\op}{\CATlarge}$ is the $2$-category of pseudofunctors, oplax natural transformations and modifications.
\end{prop}
\begin{proof}
	Notice that
	$$\HomC{\moplax{\A\op}{\CATlarge}}{\Delta 1}{\m{F(-)}{\Catzero}}\iso \HomC{\mPsoplax{\A\op}{\CATlarge}}{\Delta 1}{\m{F(-)}{\Catzero}}$$
	So it suffices to exhibit an equivalence
	$$\m{F(-)}{\Catzero}\simeq \Fib[o][n][F(-)]$$
	in the (large) $2$-category $\mPsoplax{\A\op}{\CATlarge}$. Indeed, for a general (large) $2$-category, postcomposing with a morphism that is an equivalence in the $2$-category gives a functor between hom-categories that is an equivalence of categories. The left hand side is certainly a $2$-functor, while the right hand side is a pseudofunctor by \prox\ref{propopfib-}. We have that the Grothendieck construction gives a pseudonatural adjoint equivalence
	$$\G{-}\:\m{-}{\Catzero}\simeq\Fib[o][n][-],$$
	by \prox\ref{propgrothconstrispseudonat}. Whiskering it with $F\op$ on the left gives another pseudonatural adjoint equivalence, that is then also an equivalence in the $2$-category $\mPsoplax{\A\op}{\CATlarge}$ as needed. The quasi-inverse is given by extending to a pseudonatural transformation the quasi-inverses of the Grothendieck construction on every component. This can always be done by choosing as structure $2$-cells the pasting of the inverse of the structure $2$-cells of the Grothendieck construction with unit and counit of the adjoint equivalences on components. The triangular equalities then guarantee that we have an equivalence in $\mPsoplax{\A\op}{\CATlarge}$ (we have that the two composites are isomorphic to the identity).
	
	We now prove that the equivalence of categories 
	$$\G{F(-)}\c -\:\HomC{\moplax{\A\op}{\CATlarge}}{\Delta 1}{\m{F(-)}{\Catzero}}\simeq \HomC{\mPsoplax{\A\op}{\CATlarge}}{\Delta 1}{\Fib[o][\CAT][F(-)]}$$
	that we have produced is pseudonatural in $F$.
	
	We show that $\HomC{\mPsoplax{\A\op}{\CATlarge}}{\Delta 1}{\Fib[o][\CAT][+(-)]}\:{\m{\A}{\CAT}}\op\to \CATlarge$ is a pseudofunctor. Given $\alpha\:F'\to F$ in $\m{\A}{\CAT}$, we define the image on $\alpha$ to be $\alpha_{-}\st\c -$, where
	$$\alpha_{-}\st\:\Fib[o][n][F(-)]\apseudo{}\Fib[o][n][F'(-)]\:\A\op\to \CATlarge$$
	is the pseudonatural transformation described as follows. For every $A\in \A$, we define $(\alpha_{-}\st)_A\deq \alpha_A\st$ (see \prox\ref{propopfib-}). For every $h\:A\to B$ in $\A$, we define the structure $2$-cell $(\alpha_{-}\st)_h$ to be the pasting
	\begin{cd}[5.9][5.9]
		\Fib[o][n][F(B)]\arrow[r,iso,shift right=7ex,"{}",start anchor={[xshift=-11.2ex]}] \arrow[r,iso,shift right=3.5ex,"{}",end anchor={[xshift=11.2ex]}]\arrow[rd,"{{(F(h)\c \alpha_A)}\st}"{description}] \arrow[d,"{F(h)\st}"']\arrow[r,"{\alpha_B\st}"] \& \Fib[o][n][F'(B)]\arrow[d,"{F'(h)\st}"]\\
		\Fib[o][n][F(A)] \arrow[r,"{\alpha_A\st}"'] \& \Fib[o][n][F'(A)]
	\end{cd}
	where the two isomorphisms are the ones given by the pseudofunctoriality of $\Fib[o][n][-]$ (see \prox\ref{propopfib-}).	We are using that $F(h)\c \alpha_A=\alpha_B\c F'(h)$ by naturality of $\alpha$. Then $\alpha_{-}\st$ is a pseudonatural transformation because $\Fib[o][n][-]$ is a pseudofunctor. As $\alpha_{-}\st$ is a morphism in the (large) $2$-category $\mPsoplax{\A\op}{\CATlarge}$, we have that $\alpha_{-}\st\c -$ is a functor. Considering $F''\ar{\alpha'}F'\ar{\alpha}F$ in $\m{\A}{\CAT}$, there is a an invertible modification
	$${(\alpha'_{-})}\st \c \alpha_{-}\st\iso {(\alpha\c \alpha')}_{-}\st$$
	with components given by the pseudofunctoriality of $\Fib[o][n][-]$. And then whiskering with this gives the natural isomorphism that regulates the image on the composite $\alpha\c \alpha'$.
	
	Given $\delta\:\alpha\aR{}\beta\:F'\to F$ in $\m{\A}{\CAT}$, we define the image on $\delta$ to be $\delta_{-}\st\star -$, where $\delta_{-}\st$ is the modification that has components $\delta_A\st$ on every $A\in \A$ (see \prox\ref{propopfib-}). This forms indeed a modification by pseudofunctoriality of $\Fib[o][n][-]$. It is straightforward to check that $\HomC{\mPsoplax{\A\op}{\CATlarge}}{\Delta 1}{\Fib[o][\CAT][+(-)]}$ is a pseudofunctor.
	
	We prove that $\G{F(-)}\c -$ is pseudonatural in $F\in {\m{\A}{\CAT}}\op$. Given $\alpha\:F'\to F$ in $\m{\A}{\CAT}$, we define the structure $2$-cell on $\alpha$ to be $\G{\alpha_{-}}\star -$, where $\G{\alpha_{-}}$ is the invertible modification
	\begin{cd}[5.6][5.6]
		\m{F(-)}{\Catzero} \arrow[r,simeq,"{\G{F(-)}}"]\arrow[d,"{-\c \alpha}"']\& \Fib[o][n][F(-)] \arrow[d,"{\alpha_{-}\st}"]\\
		\m{F'(-)}{\Catzero} \arrow[ru,iso,"{\G{\alpha_{-}}}"{inner sep=0.5ex}]\arrow[r,simeq,"{\G{F'(-)}}"']\& \Fib[o][n][F'(-)]
	\end{cd} 
	with components defined by the pseudonaturality of the Grothendieck construction in the base (see \prox\ref{propgrothconstrispseudonat}). The latter pseudonaturality also guarantees that $\G{\alpha_{-}}$ is a modification. Whence $\G{\alpha_{-}}\star -$ is a natural isomorphism. We then conclude that $\G{F(-)}\c -$ is pseudonatural in $F\in {\m{\A}{\CAT}}\op$ because the needed equalities of modifications can be checked on components, where everything holds because the Grothendieck construction is pseudonatural in the base (we also need the $2$-dimensional condition of this pseudonaturality).
\end{proof}

\begin{rem}
	An object of $\HomC{\mPsoplax{\A\op}{\CATlarge}}{\Delta 1}{\Fib[o][\CAT][F(-)]}$ is essentially a collection of opfibrations on every index $A\in \A$ together with a compact information on how to move between different indexes. The last ingredient that we need in order to prove our main result is that we can pack these data in terms of an opfibration in $\m{\A}{\CAT}$ over $F$.
\end{rem}

\begin{prop}\label{propopfibsonindexes}
	There is an isomorphism of categories
	$$\HomC{\mPsoplax{\A\op}{\CATlarge}}{\Delta 1}{\Fib[o][\CAT][F(-)]}\iso \Fib[o][\m{\A}{\CAT}][F]$$
	which is pseudonatural in $F$ \pteor{with the structure 2-cells being identities}.
\end{prop}
\begin{proof}
	Given $\phi\:G\to F$ an opfibration in $\m{\A}{\CAT}$, we produce an oplax natural
	$$[\phi]\:\Delta 1\aoplax{}{\Fib[o][\CAT][F(-)]}\:\A\op\to \CATlarge.$$
	For every $A\in \A$, we define $[\phi]_A\deq \phi_A$, thanks to \prox\ref{propcharactopfibacat}. For every $h\:A\to B$ in $\A$, the structure $2$-cell $[\phi]_h$ is the functor $\phi_A\to F(h)\st (\phi_B)$ defined by the universal property of the pullback in $\CAT$:
	\pbsqunivv[pos=0.3]{G(A)}{G(h)}{\phi_A}{[\phi]_h}{F(h)\st (G(B))}{G(B)}{F(A)}{F(B)}{}{F(h)\st (\phi_B)}{\phi_B}{F(h)}
	$[\phi]_h$ is cleavage preserving because $(G(h),F(f))$ and the universal square that exhibits the pullback are cleavage preserving, thanks to \prox\ref{propcharactopfibacat}. $[\phi]$ is an oplax natural transformation by the universal property of the pullback, using also the pseudofunctoriality of the pullback.

	Given
	$$\gamma\:\Delta 1\aoplax{}{\Fib[o][\CAT][F(-)]},$$
	we produce an opfibration $\widehat{\gamma}\:G\to F$ in $\m{\A}{\CAT}$. We define the ($2$-)functor $G$ sending $A\in \A$ to $\dom(\gamma_A)$ and $h\:A\to B$ to the composite above of the diagram
	\begin{cd}[5.5][5.5]
		\dom(\gamma_A)\arrow[r,"{\gamma_h}"]\arrow[d,"{\gamma_A}"'] \& F(h)\st \dom(\gamma_B)\arrow[r,"{}"]\PB{rd}\arrow[d,"{F(h)\st (\gamma_B)}"']\& \dom(\gamma_B)\arrow[d,"{\gamma_B}"] \\
		F(A)\arrow[r,equal,"{}"] \& F(A)\arrow[r,"{F(h)}"']\& F(B)
	\end{cd}
	$G$ is a functor because $\gamma$ is oplax natural. For every $A\in \A$, we define $\widehat{\gamma}_A\deq\gamma_A$. Then $\widehat{\gamma}$ is a natural transformation by construction of $G$. And the naturality squares of $\widehat{\gamma}$ are cleavage preserving because every $\gamma_h$ and every universal square that exhibits a pullback are cleavage preserving. By \prox\ref{propcharactopfibacat}, we conclude that $\widehat{\gamma}$ is a split opfibration in $\m{\A}{\CAT}$.
	
	We can extend both constructions to functors, that will be inverses of each other. Given a cleavage preserving morphism $\xi\:\phi\to \psi$ between split opfibrations in $\m{\A}{\CAT}$ over $F$, we produce a modification $[\xi]\:[\phi]\aM{}[\psi]$. For every $A\in \A$, we define $[\xi]_A\deq \xi_A$, thanks to \prox\ref{propcharactmoropfibacat}. It is straightforward to prove that this is a modification using the universal property of the pullback. Then $[-]$ is readily seen to be a functor, because the conditions can be checked on components. Given a modification
	$$\zeta\:\gamma\aM{}\delta\:\Delta 1\aoplax{}{\Fib[o][\CAT][F(-)]},$$
	we produce a cleavage preserving morphism $\widehat{\zeta}\:\widehat{\gamma}\to \widehat{\delta}$. For every $A\in \A$, we define $\widehat{\zeta}_A\deq \zeta_A$, and this is then clearly cleavage preserving. $\widehat{\zeta}$ is a natural transformation because $\zeta$ is a modification. By \prox\ref{propcharactmoropfibacat}, we conclude that $\widehat{\zeta}$ is a cleavage preserving morphism. Then $\widehat{-}$ is readily seen to be a functor because the conditions can be checked on components. It is straightforward to check that $[-]$ and $\widehat{-}$ are inverses of each other.

	We now prove the pseudonaturality in $F$, with the structure 2-cells being identities, of the isomorphism of categories we have just produced. The left hand side extends to a pseudofunctor by the proof of \prox\ref{propgrothconstroneveryindex}, while the right hand side extends to a pseudofunctor by \prox\ref{propopfib-}. Given a morphism $\alpha\:F'\to F$ in $\m{\A}{\CAT}$, we show that the following square is commutative:
	\begin{cd}[4][3.5]
		{\Fib[o][\m{\A}{\CAT}][F]} \arrow[r,aiso,"{}"]\arrow[d,"{\alpha\st}"']\& {\HomC{\mPsoplax{\A\op}{\CATlarge}}{\Delta 1}{\Fib[o][\CAT][F(-)]}} \arrow[d,"{\alpha_{-}\st\c -}"]\\
		{\Fib[o][\m{\A}{\CAT}][F']}\arrow[r,aiso,"{}"] \& 	{\HomC{\mPsoplax{\A\op}{\CATlarge}}{\Delta 1}{\Fib[o][\CAT][F'(-)]}}{}
	\end{cd}
	Let $\phi\:G\to F$ be a split opfibration in $\m{\A}{\CAT}$. For every $A\in \A$, since pullbacks in $\m{\A}{\CAT}$ are calculated pointwise,
	$$(\alpha_{-}\st\c [\phi])_A=\alpha_A\st(\phi_A)=(\alpha\st \phi)_A=[\alpha\st \phi]_A.$$
	For every $h\:A\to B$ in $\A$, we have that $[\alpha\st \phi]_h$ is equal to the pasting
	\begin{cd}[5.9][5.9]
		\1 \arrow[d,equal,"{}"]\arrow[r,"{\phi_B}"] \& \Fib[o][n][F(B)]\arrow[r,iso,shift right=7ex,"{}",start anchor={[xshift=-11.2ex]}] \arrow[r,iso,shift right=3.5ex,"{}",end anchor={[xshift=11.2ex]}]]\arrow[rd,"{{(F(h)\c \alpha_A)}\st}"{description}] \arrow[d,"{F(h)\st}"']\arrow[r,"{\alpha_B\st}"] \& \Fib[o][n][F'(B)]\arrow[d,"{F'(h)\st}"]\\
		\1 \arrow[ru,Rightarrow,"{[\phi]_h}"{pos=0.585,inner sep=0.35ex},shorten <=3.5ex,shorten >=3ex]\arrow[r,"{\phi_A}"'] \& \Fib[o][n][F(A)] \arrow[r,"{\alpha_A\st}"'] \& \Fib[o][n][F'(A)]
	\end{cd}
	by the universal property of the pullback, using again that pullbacks in $\m{\A}{\CAT}$ are calculated pointwise. It is then easy to see that the square above is commutative on morphisms $\xi\:\phi\to \psi$ as well, since it can be checked on components $A\in \A$.
	
	We prove that identities are the structure $2$-cells of an isomorphic pseudonatural transformation
	$$\Fib[o][\m{\A}{\CAT}][+] \iso \HomC{\mPsoplax{\A\op}{\CATlarge}}{\Delta 1}{\Fib[o][\CAT][+(-)]}\:{\m{\A}{\CAT}}\op\to \CATlarge$$
	This means that the isomorphisms that regulate the image of the two pseudofunctors on identities and compositions are compatible, and that the two pseudofunctors agree on $2$-cells. The first condition holds because it can be checked on components and pullbacks in $\m{\A}{\CAT}$ are calculated pointwise (choosing pullbacks along identities to be the identity). The second condition is, for every $\delta\:\alpha\aR{}\beta\:F'\to F$ in $\m{\A}{\CAT}$,
	$$[-]\star \delta\st=(\delta_{-}\st \star -)\star [-]$$ 
	This can be checked on components $\phi\:G\to F$ (split opfibration in $\m{\A}{\CAT}$). On such components we need to prove an equality of modifications, that can be then checked on components $A\in \A$. So we need to show
	$$(\delta\st \phi)_A=\delta_A\st (\phi_A).$$
	This holds because the components of the liftings along $\phi$ are the liftings along the components of $\phi$. Indeed, for a general $\theta$ as below,
	\begin{cd}[4.3][4.5]
		\y{A}\arrow[r,"{x}"]\arrow[dd,equal,"{}"]\& X \arrow[dd,equal]\arrow[rr,bend left=25,"{\alpha}",""'{name=A}]\arrow[rr,bend right=25,"{\theta\stb \alpha}"',""{name=B}]\& \& G \arrow[dd,"{\phi}"]\\[1ex]
		\&\& |[alias=W]| G \arrow[rd,bend left=20,"{\phi}"]\\[-5ex]
		\y{A}\arrow[r,"{x}"]\&X \arrow[ru,bend left=18,"{\alpha}"]\arrow[rr,bend right=18,"{\beta}"',""{name=Q}] \&\hphantom{.}\& F
		\arrow[from=A,to=B,Rightarrow,"{\h[2]\ov{\theta}^\alpha}"]
		\arrow[from=W,to=Q,Rightarrow,"{\h[2]\theta}"]
	\end{cd}
	the second condition of \defx\ref{defopfibinatwocat} ensures that the lifting of $\theta_{A,x}$ along $\phi$ is equal to $\ov{\theta}^\alpha_{A,x}$. But the former is also the lifting of $\theta_{A,x}$ (seen as a morphism in $F(A)$) along $\phi_A$, thanks to \prox\ref{propcharactopfibacat}. And everything works on morphisms $f\:x\to x'$ in $X(A)$ as well by cartesianity arguments, using the naturality of $\ov{\theta}^\alpha_A$.
\end{proof}

We are now ready to prove our main result.

\begin{teor}\label{teormain}
	Let $\A$ be a small category and consider the functor $2$-category $\m{\A}{\CAT}$. For every $2$-functor $F\:\A\to \CAT$, there is an equivalence of categories
	$$\Fib[o][\m{\A}{\CAT}][F]\simeq \m{\Grothdiag{F}}{\CAT}$$
	between split opfibrations in $\m{\A}{\CAT}$ over $F$ and $2$-copresheaves on the Grothendieck construction $\Groth{F}$ of $F$. Moreover this equivalence is pseudonatural in $F$.
\end{teor}
\begin{proof}
	It suffices to compose the equivalences of categories of \prox\ref{propapplyoplaxcolim}, \prox\ref{propgrothconstroneveryindex} and \prox\ref{propopfibsonindexes}. 
	\begin{center}
		\linesep{1.9}
		\begin{tabular}{LL}
			\m{\Grothdiag{F}}{\Catzero} \iso \HomC{\moplax{\A\op}{\CATlarge}}{\Delta 1}{\m{F(-)}{\Catzero}}\simeq\\
			\simeq \HomC{\mPsoplax{\A\op}{\CATlarge}}{\Delta 1}{\Fib[o][\CAT][F(-)]}\iso \Fib[o][\m{\A}{\CAT}][F]
		\end{tabular}
	\end{center}
	Notice that a $2$-functor from a category into $\CAT$ is the same thing as its underlying functor. As all three equivalences are pseudonatural in $F\in \m{\A}{\CAT}\op$, so is the composite.
\end{proof}

We can extract the explicit indexed Grothendieck construction from the proof of \thex\ref{teormain}.

\begin{cons}[Indexed Grothendieck construction]\label{consindexedgroth}
	We can follow the chain of equivalences of the proof of \thex\ref{teormain} to get its explicit action. Let $\phi\:G\to F$ be a split opfibration in $\m{\A}{\CAT}$ over $F$. We first produce the oplax natural transformation
	$$[\phi]\:\Delta 1\aoplax{}{\Fib[o][\CAT][F(-)]}$$
	with 
	$[\phi]_A\deq \phi_A$ for every $A\in \A$ and
	\pbsqunivv[pos=0.3]{G(A)}{G(h)}{\phi_A}{[\phi]_h}{F(h)\st (G(B))}{G(B)}{F(A)}{F(B)}{}{F(h)\st (\phi_B)}{\phi_B}{F(h)}
	for every $h\:A\to B$ in $\A$. Then we produce the oplax natural transformation with components on $A\in\A$ and structure $2$-cells on $h\:A\to B$ defined by the pasting
	\begin{cd}[5.9][5.9]
		\1 \arrow[r,"{\phi_B}"]\arrow[d,equal,"{}"]\& \Fib[o][n][F(B)]\arrow[d,"{F(h)\st}"]\arrow[r,"{\Gprime{F(B)}}"] \&[-0.4ex] \m{F(B)}{\Catzero} \arrow[d,"{-\c F(h)}"]\\
		\1 \arrow[ru,Rightarrow,"{[\phi]_h}"{pos=0.585,inner sep=0.35ex},shorten <=3.5ex,shorten >=3ex]\arrow[r,"{\phi_A}"']\& \Fib[o][n][F(A)]\arrow[r,"{\Gprime{F(A)}}"']\arrow[ru,iso,"{\Gprime{F(h)}}"'] \& \m{F(A)}{\Catzero}
	\end{cd}
	where $\Gprime{-}$ is the quasi-inverse of the Grothendieck construction. We have that $\Gprime{F(A)}(\phi_A)$ sends every $X\in F(A)$ to the fibre ${(\phi_A)}_X$ of $\phi_A$ over $X$ and every morphism $\alpha\:X\to X'$ in $F(A)$ to the functor
	$$\alpha\stb\:{(\phi_A)}_X\to {(\phi_A)}_{X'}$$
	that lifts $\alpha$ (on morphisms, it is defined by cartesianity). The structure $2$-cell on $h$ is the natural transformation with component on $X\in F(A)$ given by
	$${(\phi_A)}_X\ar{[\phi]_h}(F(h)\st(\phi_B))_X\iso (\phi_B)_{F(h)(X)}$$
	which coincides with $G(h)$.
	
	\noindent Finally, we induce the $2$-functor\v[0.5]
	\begin{fun}
		\grothprime{\phi} & \: & \Grothdiag{F} \hphantom{CCC}& \too &\hphantom{CC} \CAT \\[1.4ex]
		&&\begin{cd}*[3][3]
			(A,X) \arrow[dd,bend right=70,"{(h,\alpha)}"']\arrow[d,"{(h,\id{})}"]\\
			(B,F(h)(X))\arrow[d,"{(\id{},\alpha)}"]\\
			(B,X')
		\end{cd} & \mto &  \begin{cd}*[3][3]
			{(\phi_A)}_X \arrow[d,"{G(h)}"]\\
			{(\phi_B)}_{F(h)(X)}\arrow[d,"{\alpha\stb}"] \\
			{(\phi_B)}_{X'}
		\end{cd}
	\end{fun}
	using the universal property of the oplax colimit $\Groth{F}$, as in the proof of \thex\ref{teorgrothisoplaxcolim}. 
	
	Let $Z\:\Groth{F}\to \CAT$ be a $2$-functor. We produce the oplax natural transformation $\gamma$ with components on $A\in \A$ and structure $2$-cells on $h\:A\to B$ in $\A$ defined by the pasting
	\begin{cd}[5.9][5.9]
		\1 \arrow[r,"{Z\c \opn{inc}_{\h[1]B}}"]\arrow[d,equal,"{}"]\& \m{F(B)}{\Catzero} \arrow[d,"{-\c F(h)}"] \arrow[r,"{\G{F(B)}}"]\&[-0.4ex]
		\Fib[o][n][F(B)]\arrow[d,"{F(h)\st}"] \\
		\1 \arrow[ru,Rightarrow,"{Z\star \opn{inc}_h}"{pos=0.585,inner sep=0.35ex},shorten <=3.5ex,shorten >=3ex]\arrow[r,"{Z\c \opn{inc}_{\h[1]A}}"']\&
		\m{F(A)}{\Catzero}\arrow[r,"{\G{F(A)}}"']\arrow[ru,iso,"{\G{F(h)}}"']\& \Fib[o][n][F(A)] 
	\end{cd}
	We have that $\G{F(A)}\c Z\c \opn{inc}_{\h[1]A}\:G(A)\to F(A)$ is the Grothendieck construction of the $2$-functor $F(A)\to \CAT$ that sends every object $X\in F(A)$ to $Z(A,X)$ and every morphism $\alpha\:X\to X'$ in $F(A)$ to $Z(\id{},\alpha)$. Its domain $G(A)$ has the following description:
	\begin{description}
		\item[an object] is a pair $(X,\xi)$ with $X\in F(A)$ and $\xi\in Z(A,X)$;
		\item[a morphism $(X,\xi)\to (X',\xi')$] is a pair $(\alpha,\Xi)$ with $\alpha\:X\to X'$ in $F(A)$ and\\ $\Xi\:Z(\id{},\alpha)(\xi)\to \xi'$ in $Z(A,X')$.
	\end{description}
	These are then collected as a split opfibration $\groth{Z}\:G\to F$ in $\m{\A}{\CAT}$ over $F$ whose components on every $A$ are the projections $G(A)\to F(A)$ on the first component. For every $h\:A\to B$ in $\A$, the functor $G(h)$ is defined by the composite above in the diagram
	\begin{cd}[5.5][5.5]
		G(A)\arrow[r,"{\gamma_h}"]\arrow[d,"{\groth{Z}_A}"'] \& F(h)\st G(B)\arrow[r,"{}"]\PB{rd}\arrow[d,"{F(h)\st (\groth{Z}_B)}"']\& G(B)\arrow[d,"{\groth{Z}_B}"] \\
		F(A)\arrow[r,equal,"{}"] \& F(A)\arrow[r,"{F(h)}"']\& F(B)
	\end{cd}
	Explicitly,
	\begin{fun}
		G(h) & \: & G(A)\hphantom{C..} & \too & \hphantom{CC.}G(B) \\[1ex]
		&& \fib{(X,\xi)}{(\alpha,\Xi)}{(X',\xi')} & \mto & \fib{(F(h)(X),Z(h,\id{})(\xi))}{(F(h)(\alpha),Z(h,\id{})(\Xi))}{(F(h)(X'),Z(h,\id{})(\xi'))}
	\end{fun}
	
	We can see how this construction is indeed an indexed Grothendieck construction. We essentially collect together triples $(A,X,\xi)$ with $A\in \A$, $X\in F(A)$ and $\xi\in Z(A,X)$.
\end{cons}

\begin{teor}\label{teorrestrictstodisc}
	Let $\A$ be a small category. For every $2$-functor $F\:\A\to \CAT$, the equivalence of categories
	$$\Fib[o][\m{\A}{\CAT}][F]\simeq \m{\Grothdiag{F}}{\CAT}$$
	of \thex\ref{teormain} restricts to an equivalence of categories
	$$\Fib+[b][\m{\A}{\CAT}][F]\simeq \m{\Grothdiag{F}}{\Set}$$
	between discrete opfibrations in $\m{\A}{\CAT}$ over $F$ with small fibres and $\Set$-valued copresheaves on $\Groth{F}$. Moreover this equivalence is pseudonatural in $F$.
\end{teor}
\begin{proof}
	The isomorphism of \prox\ref{propapplyoplaxcolim} restricts to one with $\Set$ on both sides in the place of $\Catzero$ by $2$-naturality in $U$ of the isomorphism given by an oplax colimit (see \defx\ref{defoplaxcolim}). Pseudonaturality in $F$ still holds by the same general argument that guaranteed it with $\Catzero$ on both sides.
	
	The equivalence of \prox\ref{propgrothconstroneveryindex} restricts to one with $\Set$ in the place of $\Catzero$ on the left hand side and discrete opfibrations with small fibres in the place of opfibrations in the right hand side. Indeed the following is a commutative square of pseudonatural transformations:
	\sq[n][5][5.5]{\m{F(-)}{\Set}}{\Fib+[b][n][F(-)]}{\m{F(-)}{\Catzero}}{\Fib[o][n][F(-)]}{\G{F(-)}}{}{}{\G{F(-)}}
	On components $A\in \A$, this is true by the classical \thex\ref{usualgrothconstr}. And it is straightforward to check that it is true on structure $2$-cells as well, since structure $2$-cells are given by the pseudofunctoriality of the pullback. Then pseudonaturality in $F$ holds for the restricted equivalence as well, as one can readily check.
	
	The isomorphism of \prox\ref{propopfibsonindexes} restricts to one with discrete opfibrations with small fibres on both sides in the place of opfibrations, because it suffices to look at the components. Then pseudonaturality in $F$ holds as well, precomposing the pseudonatural transformation produced in the proof of \prox\ref{propopfibsonindexes} with the inclusion of discrete opfibrations with small fibres into opfibrations.
\end{proof}

\begin{rem}
	When $F\:\A\to \Set$, the equivalence of categories
	$$\Fib+[b][\m{\A}{\CAT}][F]\simeq \m{\Grothdiag{F}}{\Set}$$
	becomes the well known
	$$\slice{\m{\A}{\Set}}{F}\simeq \m{\Grothdiag{F}}{\Set}.$$
	Indeed any discrete opfibration $\phi\:G\to F$ in $\m{\A}{\CAT}$ over $F\:\A\to \Set$ with small fibres needs to have $G\:\A\to \Set$, and all functors $G\to F$ in $\m{\A}{\Set}$ are discrete opfibrations with small fibres. Our theorem guarantees that this equivalence is pseudonatural in $F$, which does not seem to appear in the literature.
	
	When $F$ is a representable $\y{A}\:\A\to \Set$, we obtain the famous equivalence
	$$\slice{\m{\A}{\Set}}{\y{A}}\simeq \m{\slice{\A}{A}}{\Set}$$
	between slices of (co)presheaves and (co)presheaves on slices. We will apply its pseudonaturality in $F$ in \exax\ref{exhofstreichuniv} to get a nice candidate for a Hofmann--Streicher universe (see~\cite{hofmannstreicher_liftinggrothuniverses}) in 2-presheaves. 
\end{rem}

\begin{rem}
	The equivalence
	$$\slice{\m{\A}{\Set}}{F}\simeq \m{\Grothdiag{F}}{\Set}.$$
	had many applications in geometry and logic. It is the archetypal case of the fundamental theorem of elementary topos theory, as it shows that every slice of a Grothendieck topos is a Grothendieck topos. 
	
	We can interpret our main theorem as a $2$-dimensional generalization of this. Indeed, the concept of (op)fibrational slice has recently been proposed as the correct upgrade of slices to dimension 2. This idea appears in Ahrens, North and van der Weide's~\cite{ahrensnorthvanderweide_bicategoricaltypetheory}, where it is attributed to Shulman. Our equivalence
	$$\Fib[o][\m{\A}{\CAT}][F]\simeq \m{\Grothdiag{F}}{\CAT}$$
	says that every opfibrational slice of a Grothendieck $2$-topos is again a Grothendieck $2$-topos. Notice that a morphism in a 2-category $\St[J]{\A}$ of stacks is a discrete opfibration if and only if its underlying morphism in $\m{\A\op}{\Cat}$ is so (see the second author's~\cite{mesiti_twoclassifiersdensegenstacks}).
\end{rem}

We can now explore some variations on the indexed Grothendieck construction.

\begin{rem}\label{remopvariations}
	We can change $\A$ to $\A\op$ and get the $2$-category $\m{\A\op}{\CAT}$ of $2$-presheaves. Then $F\:\A\op\to \CAT$. Be careful that, for opfibrations in $\m{\A\op}{\CAT}$, we still need to apply the Grothendieck construction to $F$ as if we did not know that the domain of $F$ is an opposite category. We write $\Grothop{F}$ for this Grothendieck construction on $F$, to emphasize that it is not the most natural one for a $2$-functor $\A\op\to \CAT$. We obtain
	$$\Fib[o][\m{\A\op}{\CAT}][F]\simeq \m{\Grothopdiag{F}}{\CAT}$$
	
	The most natural Grothendieck construction $\Groth{F}$ of a contravariant $2$-functor $F\:\A\op\to \CAT$ appears instead to handle fibrations in the place of opfibrations. Such Grothendieck construction $\Groth{F}$ is the lax colimit of $F$. Then $-\op\:\CAT\to\CAT\co$, where $\CAT\co$ is the dualization on 2-cells of $\CAT$, preserves this colimit. We obtain that $\left(\Groth{F}\right)\op$ is the lax colimit in $\CAT\co$ of $F(-)\op$, which means that
	$${\p{\Grothdiag{F}}}\op=\oplaxcolim{\left(F(-)\op\right)}$$
	in $\CAT$. Then we have the following chain of equivalences of categories:
	\begin{center}
		\linesep{1.9}
		\begin{tabular}{LL}
			\m{{\p{\Grothdiag{F}}}\op}{\Catzero} \iso \HomC{\moplax{\A}{\CATlarge}}{\Delta 1}{\m{{F(-)}\op}{\Catzero}}\simeq\\
			\simeq \HomC{\mPsoplax{\A}{\CATlarge}}{\Delta 1}{\Fib[n][\CAT][F(-)]}\iso \Fib[n][\m{\A\op}{\CAT}][F]
		\end{tabular}
	\end{center}
\end{rem}

\begin{rem}\label{remtwocatvariations}
	We believe that, when $\A$ is a $2$-category, one can still obtain an equivalence of categories 
	$$\Fib[o][\m{\A}{\CAT}][F]\simeq \m{\Grothdiag{F}}{\CAT}$$
	where $\Groth{F}$ is now the $2$-$\Set$-enriched Grothendieck construction (introduced by Street in~\cite{street_limitsindexedbycatvalued} and explored more in detail by the second author in~\cite{mesiti_twosetenrgrothconstrlaxnlimits}). In order to adapt our proof of $\thex\ref{teormain}$ to this setting, one would need $\Groth{F}$ to be a kind of oplax colimit in $\twoCAT$. We believe that $F$ is the $2$-oplax colimit of $F$ followed by the inclusion $i$ of $\CAT$ into $\twoCAT$, where a $2$-oplax natural transformation is a Crans's~\cite{crans_tensorprodforgraycat} oplax $1$-transfor. Such transformations have the same $1$-dimensional conditions of an oplax natural transformation but now also have structure $3$-cells on every $2$-cell in $\A$. Having as codomain $\twoCAT$, they compose well. The added structure $3$-cells are precisely what one needs in order to encode the $2$-cells
	$$\underline{\delta}^X\:(f,F(\delta)_X)\aR{}(g,\id{})\:(A,X)\to (B,F(g)(X))$$
	in $\Groth{F}$ for every $\delta\:f\aR{}g\:A\to B$ in $\A$. As explained by the second author in~\cite{mesiti_twosetenrgrothconstrlaxnlimits}, every $2$-cell in $\Groth{F}$ is a whiskering of such particular $2$-cells (in some sense, these $2$-cells are the only ones we need). The middle equivalence of the chain that proves our $\thex\ref{teormain}$ would then be given by the $2$-$\Set$-enriched Grothendieck construction. Finally, the last part of the chain would probably work as well, with the structure $3$-cells managing to encode the action of $G$ on $2$-cells. However, we have not checked these details.
	
	Such generalization would be helpful also to handle non-split opfibrations and pseudofunctors from $\Groth{F}$ into $\CAT$, for which we cannot reduce to functors into $\Catzero$. Of course, for this, one could also extend the explicit indexed Grothendieck construction.
	
	For the restriction to copresheaves and discrete opfibrations, we need to be careful that
	$$\HomC{\twoCAT}{\Grothdiag{F}}{i(\Set)}\iso \HomC{\CAT}{\pi\stb\Grothdiag{F}}{\Set}$$
	where $\pi\st$ is the left adjoint of $i\:\CAT\to \twoCAT$.
	So a quotient of $\Groth{F}$ by its $2$-cells appears: morphisms in $\Groth{F}$ that were connected via a $2$-cell becomes equal.
\end{rem}

\section{Examples}\label{sectionexamples}

In this section, we show some interesting examples. We can vary both $\A$ and $F$ in our main results. We start with $\A=\1$, that recovers the usual Grothendieck construction. $\A=\2$ represents the simultaneous Grothendieck construction of two opfibrations connected by an arrow. While $\A=\Delta$ considers (co)simplicial categories. We also explore other examples. In particular, we obtain a nice candidate for a Hofmann--Streicher universe in 2-presheaves. 

\begin{rem}[$\A=\1$]\label{exa=1}
	When $\A=\1$, the $2$-category $\m{\A}{\CAT}$ reduces to $\CAT$. A $2$-functor $F\:\1\to \CAT$ is just a small category $\C$ and $\Groth{F}=\C$. So \thex\ref{teormain} gives the classical
	$$\Fib[o][n][\C]\simeq \m{\C}{\CAT}.$$
	The explicit indexed Grothendieck construction becomes the usual Grothendieck construction. Indeed the first part and the last part of the chain become trivial, while the middle part is the Grothendieck construction on the unique index $\ast\in \1$.
\end{rem}

\begin{exampl}[$\A$ discrete]
	When $\A$ is discrete, the $2$-category $\m{\A}{\CAT}$ is a product of copies of $\CAT$. A $2$-functor $F\:\A\to \CAT$ just picks as many categories as the cardinality of $\A$, without bonds. Since the diagram $F$ is parametrized by a discrete category, we have that $\Groth{F}=\oplaxcolim{F}$ becomes the coproduct of the categories picked by $F$. And $\m{\Groth{F}}{\CAT}$ is then a collection of functors from every such category into $\CAT$. On the other hand, by \prox\ref{propcharactopfibacat}, a split opfibration in $\m{\A}{\CAT}$ is just a collection of as many opfibrations as the cardinality of $\A$, without bonds. The indexed Grothendieck construction
	$$\m{\Grothdiag{F}}{\CAT}\simeq \Fib[o][\m{\A}{\CAT}][F]$$
	is the simultaneous Grothendieck construction of all the functors into $\CAT$ that are collected as a single functor from the coproduct. This shows the indexed nature of the indexed Grothendieck construction.
\end{exampl}

\begin{exampl}[$\A=\2$]\label{exa=2}
	When $\A=\2$, the $2$-category $\m{\A}{\CAT}$ is the arrow category of $\CAT$ and $F\:\2\to \CAT$ is a functor $\t{F}\:\C\to \D$. The Grothendieck construction $\Groth{F}$ has as objects the disjoint union of the objects of $\C$ and of $\D$, denoted respectively $(0,C)$ and $(1,D)$ with $C\in \C$ and $D\in \D$. The morphisms of $\Groth{F}$ are of three kinds: morphisms in $\C$ (over $0$), morphisms in $\D$ (over $1$) and morphisms over $0\to 1$ that represents the objects $(C,D,\t{F}(C)\to D)$ of the comma category $\slice{\t{F}}{\D}$. On the other hand, given $G\:\2\to \CAT$ corresponding to $\t{G}\:\E\to \L$, a split opfibration $\phi\:G\to F$ in $\m{\2}{\CAT}$ is a cleavage preserving morphism
	\sq[n][5.75][5.75]{\E}{\L}{\C}{\D}{\t{G}}{\phi_0}{\phi_1}{\t{F}}
	between split opfibrations $\phi_0$ and $\phi_1$. So \thex\ref{teormain} gives an equivalence of categories between morphisms of (classical) split opfibrations (in $\CAT$) which have $\t{F}$ as second component and $2$-copresheaves on a category that collects together $\C$, $\D$ and the comma category $\slice{\t{F}}{\D}$.
	
	Following \conx\ref{consindexedgroth}, we get the explicit (quasi-inverse of the) indexed Grothendieck construction in this case. The arrow above between split opfibrations $\phi_0$ and $\phi_1$ can be reorganized as the functor $\Groth{F}\to \CAT$ that sends
	\begin{enum}
		\item $(0,C)$ to the fibre of $\phi_0$ on $C$ and every morphism $f$ in $\C$ to the functor $f\stb$ that lifts it along $\phi_0$;
		\item $(1,D)$ to the fibre of $\phi_1$ on $D$ and every morphism $g$ in $\D$ to the functor $g\stb$ that lifts it along $\phi_1$;
		\item every morphism corresponding to an object $(C,D,\alpha\:\t{F}(C)\to D)$ of the comma category $\slice{\t{F}}{\D}$ to the composite functor
		$${(\phi_0)}_C\ar{\t{G}} {(\phi_1)}_{\t{F}(C)}\ar{\alpha\stb}{(\phi_1)}_D.$$
	\end{enum}
	
	In the particular case in which $F\:\2\to \Set$, we have that $\t{F}\:S\to T$ is a function between sets. Then $\Groth{F}$ is a poset with objects the disjoint union of the objects of $S$ and of $T$ and such that $(0,s)\leq (1,t)$ with $s\in S$ and $t\in T$ if and only if $\t{F}(s)=t$. On the other hand, a split opfibration in $\m{\2}{\CAT}$ over $F$ is precisely a commutative square in $\Set$ with bottom leg equal to $\t{F}$.
\end{exampl}

\begin{exampl}[$\A=\I$]
	When $\A$ is the walking isomorphism $\I$, we have that $F\:\I\to \CAT$ is an invertible functor $\t{F}\:\C\to \D$. Then $\Groth{F}$ is similar to the one of \exax\ref{exa=2}, but there is now a fourth kind of morphisms, that represents the objects $(D,C,\t{F}^{-1}(D)\to C)$ of the comma category $\slice{\t{F}^{-1}}{\C}$. 
	
	If $F\:\I\to \Set$, the partial order of the poset constructed as in \exax\ref{exa=2} now becomes an equivalence relation. Every object is in relation precisely with itself and with its copy in the other set.
\end{exampl}

\begin{exampl}[$\A=\Delta$]
	When $\A$ is the simplex category $\Delta$, we have that $F\:\Delta\to \CAT$ is a cosimplicial category. This is equivalently a cosimplicial object in $\CAT$ or an internal category in cosimplicial sets. The Grothendieck construction $\Groth{F}$ collects together all the cosimplexes in a total category, taking into account faces and degeneracies. \thex\ref{teormain} gives an equivalence of categories between split opfibrations between cosimplicial categories over $F$ and functors into $\CAT$ from the total category that collects all the cosimplexes given by $F$.
\end{exampl}

\begin{exampl}[$F=\Delta 1$]
	Given any small category $\A$, we can consider $F=\Delta 1\:\A\to \CAT$ the functor constant at the terminal $\1$. We have that $\Groth{\Delta 1}=\A$. So \thex\ref{teormain} gives an equivalence of categories
	$$\Fib[o][\m{\A}{\CAT}][\Delta 1]\simeq \m{\A}{\CAT}.$$
	Indeed, as being opfibred over $\1$ means nothing, a split opfibration $\phi\:G\to \Delta 1$ is a collection of categories $G(A)$ and of functors $G(h)$ for every $h\:A\to B$ in $\A$. This forms a functor $\A\to \CAT$ because $\phi$ is split.
	
	Putting together this equivalence with that of \exax\ref{exa=1}, we obtain
	$$\Fib[o][\m{\A}{\CAT}][\Delta 1]\simeq \Fib[o][\CAT][\A]$$
\end{exampl}

\begin{exampl}[$F=\Delta \B$]
	Given any small category $\A$, we can consider $F=\Delta \B\:\A\to \CAT$ the functor constant at a fixed category $\B$. We have that $\Groth{\Delta \B}=\A\times \B$ and $\groth{\Delta \B}$ is the projection $\A\times \B\to \A$. \thex\ref{teormain} characterizes functors $\A\times \B\to \CAT$, and hence the $\Cat$-enriched profunctors, in terms of split opfibrations in $\m{\A}{\CAT}$ over $\Delta \B$.
\end{exampl}

\begin{exampl}[semidirect product of groups]
	Let $\A$ be the one-object category $\mathcal{B}G$ corresponding to a group $G$. Consider then $F\:\mathcal{B}G\to \CAT$ that sends the unique object of $\mathcal{B}G$ to the one-object category that corresponds to a group $H$. Functoriality of $F$ corresponds precisely to giving a group homomorphism $\rho\:G\to \opn{Aut}(H)$ where $\opn{Aut}(H)$ is the group of automorphisms of $H$. Then the Grothendieck construction $\Groth{F}$ is a one-object category corresponding to the semidirect product $H\rtimes_\rho G$. Thus \thex\ref{teormain} characterizes functors $H\rtimes_\rho G\to \CAT$ in terms of opfibrations in $\m{\mathcal{B}G}{\CAT}$ over the functor $F$ that corresponds with $\rho\:G\to \opn{Aut}(H)$.
\end{exampl}
	
\begin{exampl}[Hofmann--Streicher universe in $2$-presheaves]\label{exhofstreichuniv}
	We apply \thex\ref{teormain} to get a nice candidate for a Hofmann--Streicher universe (see~\cite{hofmannstreicher_liftinggrothuniverses}) in the $2$-category $\m{\A\op}{\CAT}$ of $2$-presheaves. The second author has shown in the following paper~\cite{mesiti_twoclassifiersdensegenstacks} that such candidate is indeed a $2$-dimensional classifier in $\m{\A\op}{\CAT}$, towards a $2$-dimensional elementary topos structure on $\m{\A\op}{\CAT}$. This was actually the starting motivation for the second author to produce the indexed Grothendieck construction.
	
	Recall that the archetypal subobject classifier is $T:\1\to \{T,F\}$ in $\Set$. Every subset $A\cont X$ is classified by its characteristic function $\chi_A\:X\to \{T,F\}$. More precisely, pulling back $T:\1\to \{T,F\}$ gives a bijection
	$$\m{X}{\{T,F\}}\iso \Sub{X}$$
	for every $X\in \Set$, exhibiting $\Set$ as the archetypal elementary topos. Notice that injective functions have as fibres either the empty set or the singleton; we could say that we are classifying morphisms with fibres of dimension $0$.
	
	Moving to dimension $2$, the 2-category $\CAT$ becomes the archetypal elementary $2$-topos. And we now want to classify morphisms with fibres of dimension $1$, i.e.\ fibres that are general sets. As proposed by Weber in~\cite{weber_yonfromtwotop}, the correct $2$-categorical generalization of subobject classifiers are classifiers of discrete opfibrations. The archetypal 2-dimensional classifying process is, in $\CAT$, the construction of the category of elements (i.e.\ the Grothendieck construction restricted to functors into $\Set$). This can be captured as a comma object from $\1\:\1\to \Set$ or as a pullback of the forgetful $\Set\b\to \Set$ from pointed sets to sets, on the line of \thex\ref{teorgrothconstrislaxcomma}. Doing either of the two provides an equivalence of categories
	$$\groth{-}\:\m{\C}{\Set}\aequi \Fib+[b][n][\C]$$
	for every $\C\in \CAT$, exhibiting $\Set$ as the $2$-dimensional universe of the elementary $2$-topos $\CAT$. Looking at the archetypal 2-dimensional classifying process, we can think of a 2-classifier as a Grothendieck construction inside a 2-category. So it is natural to expect an indexed version of the Grothendieck construction to give a 2-classifier in the 2-category of  2-presheaves.
	
	Given a category $\A$, we would like to produce a $2$-dimensional universe $\O$ in the $2$-category $\m{\A\op}{\CAT}$ of $2$-presheaves. We then want an equivalence of categories
	$$\groth{-}\:\m{F}{\O}\aequii \Fib+[b][\m{\A\op}{\CAT}][F]$$
	 for every $F\in \m{\A\op}{\CAT}$. In particular, this needs to hold for representables $F=\y{A}$ with $A\in \A$. So by Yoneda's lemma, we want
	 $$\O(A)\simeq \Fib+[b][\m{\A\op}{\CAT}][\y{A}]$$
	 Trying to define $\O$ to send $A\in \A$ precisely to $\Fib+[b][\m{\A\op}{\CAT}][\y{A}]$, we only get a pseudofunctor $\O'$, defined by \prox\ref{propopfib-}, that does not even clearly land in small categories. \thex\ref{teormain} (together with \remx\ref{remopvariations}) offers a nice way to replace such pseudofunctor $\O'$ with a strict $2$-functor $\O$, that moreover lands in small categories. Indeed it gives an equivalence of categories 
	$$\Fib+[b][\m{\A\op}{\CAT}][\y{A}]\simeq \m{\Grothopdiag{\y{A}}}{\Set}$$
	that is pseudonatural in $A\in \A\op$, by precomposing the equivalence that is pseudonatural in $F\in {\m{\A\op}{\CAT}}\op$ with $\yyop\:\A\op \to {\m{\A\op}{\CAT}}\op$. So the right hand side of the equivalence above gives a strict $2$-functor $\O$ that is pseudonaturally equivalent to $\O'$.
	
	When $\A$ is a category, $$\Grothopdiag{\y{A}}={\left(\slice{\A}{A}\right)}\op \quad\text{ and }\quad \O(A)=\m{{\left(\slice{\A}{A}\right)}\op}{\Set}.$$
	$\O$ acts on morphisms by postcomposition. Notice that in this case $\y{A}\:\A\op\to \Set$ and so the left hand side of the equivalence above simplifies to $\slice{\m{\A\op}{\Set}}{\y{A}}$, but we still need the pseudonaturality in $A$, that does not seem to appear in the literature. Our \thex\ref{teormain} guarantees such pseudonaturality and therefore that we get a strict $2$-functor $\O$ pseudonaturally equivalent to $\O'$. This is a Hofmann--Streicher universe, in line with the ideas of~\cite{hofmannstreicher_liftinggrothuniverses} and with Awodey's recent work~\cite{awodey_onhofmannstreicheruniverses}. The second author has shown in the following paper~\cite{mesiti_twoclassifiersdensegenstacks} that $\O$ is indeed a $2$-classifier in $\m{\A\op}{\CAT}$, by an argument of reduction of the study of $2$-classifiers to dense generators. He has then restricted this $2$-classifier to one in stacks.
	
	When $\A$ is a $2$-category, the $2$-$\Set$-enriched Grothendieck construction (introduced by Street in~\cite{street_limitsindexedbycatvalued} and explored by the second author in~\cite{mesiti_twosetenrgrothconstrlaxnlimits}) gives
	$$\Grothopdiag{\y{A}}={\left(\oplaxslice{\A}{A}\right)}\op$$
	Checking the details of the strategy proposed in \remx\ref{remtwocatvariations}, we would get a refined strict $2$-functor $\O$ defined by
	$$\O(A)=\m{\pi\stb{\left(\oplaxslice{\A}{A}\right)}\op}{\Set}.$$
	Interestingly, such quotients of (op)lax slices give the right weights to represent (op)lax (co)limits as weighted ones, by Street's~\cite{street_limitsindexedbycatvalued}.
\end{exampl}

\subsection*{Acknowledgements}

We would like to thank the anonymous referee and the editor for their helpful comments and suggestions.

\bibliographystyle{abbrv}
\bibliography{Bibliography}

\end{document}